\documentclass[reqno]{amsart}[12pt]
\usepackage{enumerate}
\usepackage{hyperref}
\hypersetup{colorlinks,linkcolor={blue},citecolor={blue},urlcolor={red}}  
\usepackage[paper=letterpaper,margin=.9in]{geometry} 
\usepackage[utf8]{inputenc}
\usepackage{lmodern}
\usepackage{array}
\usepackage{adjustbox}
\usepackage{graphicx}
\usepackage{amssymb, amsfonts, amsthm, amsmath}
\usepackage{float}
\usepackage{comment}
\usepackage{tikz}
\usepackage{bbm}
\usetikzlibrary{decorations.markings}
\usetikzlibrary{calc,angles,positioning}
\usetikzlibrary{decorations.pathreplacing}
\usepackage{chngcntr}
\usepackage{apptools}
\theoremstyle{plain}
\newtheorem{theorem}{Theorem}[section]

\newtheorem{defin}[theorem]{Definition}
\newtheorem{prop}[theorem]{Proposition}

\newtheorem{lemma}[theorem]{Lemma}

\newtheorem{remark}[theorem]{Remark}

\newcommand\numberthis{\addtocounter{equation}{1}\tag{\theequation}}
\newcommand{\argmax}{\operatornamewithlimits{\textrm{argmax}}}

\newcommand{\red}[1]{\textcolor{black}{#1}}

\newcommand{\q}{\mathfrak{q}}
\newcommand{\e}{\epsilon}
\begin{document}
\numberwithin{equation}{section}
\title{Lyapunov exponents of the SHE for general initial data}
\author{Promit Ghosal \and
Yier Lin}
	\address{P.\ Ghosal,
	Department of Mathematics, Massachusetts Institute of Technology, 77 Massachusetts Avenue
Cambridge, MA 02139-4307, U.S.A}
\email{promit@mit.edu}
	\address{Y.\ Lin,
	Department of Statistics, the University of Chicago,
	5747 S Ellis Ave, Chicago, IL 60637, U.S.A}
\email{ylin10@uchicago.edu}
\begin{abstract}
We consider the $(1+1)$-dimensional stochastic heat equation (SHE) with multiplicative white noise and the Cole-Hopf solution of the Kardar-Parisi-Zhang (KPZ) equation. We show an exact way of computing the Lyapunov exponents of the SHE for a large class of initial data which includes any bounded deterministic positive initial data and the stationary initial data. As a consequence, we derive exact formulas for the upper tail large deviation rate functions of the KPZ equation for general initial data.
\end{abstract}
\maketitle
\section{Background and Main result}
In this paper, we consider the solution of the $(1+1)$-dimensional SHE under general initial condition and ask how any positive moment of that solution grows as time goes to $\infty$. In particular, we take the logarithm of $p$-th moment of that solution for any $p\in \mathbb{R}_{>0}$ and show that the when scaled by time, those converge. Those limits are known as \emph{Lyapunov exponents} which are tied to the large deviation problem of the KPZ equation. Namely, the upper tail large deviation of the Cole-Hopf solution of the KPZ equation (centered by $\mathrm{time}/24$) is the Legendre-Fenchel dual of the Lyapunov exponent of the stochastic heat equation. To the best of our knowledge, our result is the first to provide exact computation of the positive real Lyapunov exponents of the SHE and the upper tail large deviation of the KPZ for general initial data.        
\smallskip

\noindent Let us recall the KPZ equation, written formally as 
\begin{equation}\label{eq:KPZDef}
\partial_t \mathcal{H}(t, x) = \frac{1}{2} \partial_{xx} \mathcal{H}(t, x) + \frac{1}{2} (\partial_x \mathcal{H}(t, x))^2 + \xi(t, x), \qquad 
\mathcal{H}(0, x) = \mathcal{H}_0(x). 
\end{equation}
The KPZ equation governs growth of the interface $\mathcal{H}(t,x)$ which is subjected to a roughening by the space-time white noise $\xi$. Due to the presence of $\xi$, the solution of the KPZ equation is ill-posed. A formal solution of the KPZ equation comes from the Cole-Hopf transform given as 
\begin{align}\label{eq:ColeHopf}
\mathcal{H}(t,x) := \log(\mathcal{Z}(t, x)),
\end{align}   
where $\mathcal{Z}(t, x)$ is the solution of the SHE:
\begin{equation}\label{eq:shegeneral}
\partial_t \mathcal{Z}(t, x) = \frac{1}{2} \partial_{xx} \mathcal{Z}(t, x) + \mathcal{Z}(t, x) \xi(t, x),\qquad 
\mathcal{Z}(0, x) = \exp(\mathcal{H}_0(x)).
\end{equation}
The SHE is pervasive in the diffusion theory of particles in random environment \cite{Mol96, Kho14}, continuous directed random polymers \cite{HHF, Com17} and many other fields. The solution theory of the SHE is well known \cite{Walsh86, BC95, Cor18} via It\^{o} integral theory or martingale problem.
By \cite[Theorem 2.4]{quastel2011introduction}, for any non-negative initial data $\mathcal{Z}_0$ such that $\int_A \mathcal{Z}_0(x)dx$ defines a Radon measure for any Borel set $A\subset \mathbb{R}$, then \eqref{eq:shegeneral} has unique solution if the following condition is satisfied:
\begin{align}\label{eq:JCcondition}
\mathbb{E}\big[\int_{-n}^n \mathcal{Z}_0(x)dx\big]\leq ce^{cn}
\end{align}
for some constant $c>0$.
The logarithm in \eqref{eq:ColeHopf} is well defined due to the strict positivity of the solution of the SHE \cite{Mue91}. The Cole-Hopf solution correctly approximates discrete growth processes \cite{BG97,CT17,CGST, Lin20} and has shown to appear naturally in various renormalization and regularization schemes \cite{Hai13,GIP15,GP17}.   

The spatial derivative of the KPZ equation formally solves the stochastic Burgers equation - a continuum for the turbulence, interacting particle system driven, lattice gases and spin chains, see \cite{forster1977large, van1985excess, bertini1994stochastic,  ljubotina2019kardar, jin2020stochastic, ilievski2020superuniversality, de2020universality}.

Note that \eqref{eq:JCcondition} is satisfied for a large class of physically relevant initial conditions. This includes two important sub-classes, namely, the class of positive measurable functions on $\mathbb{R}$ with at most exponential growth in $x$ and the exponential of Brownian motion with linear drifts. The main goal of this paper is to derive Lyapunov exponents of the SHE started from a initial data coming from any of the above two sub-classes.

Fix any $f$ such that $e^{f}$ satisfies \eqref{eq:JCcondition}. 
We denote the SHE solutions started from $f$ by $\mathcal{Z}^{f}$ and the corresponding Cole-Hopf solution of the KPZ equation by $\mathcal{H}^{f}$. 
For any $p\in \mathbb{R}_{>0}$, we define the $p$-th moment Lyapunov exponent as 
\begin{align}
\mathrm{Lya}_{p}(f) := \lim_{t\to \infty} \frac{1}{t}\log \mathbb{E}\Big[\big(\mathcal{Z}^{f}(t,0)\big)^p\Big].
\end{align}
For the KPZ equation, we consider the upper tail probability $\mathbb{P}(\mathcal{H}^{f}(t,0)+\frac{t}{24}\geq st)$ where $s$ is a positive real number. We ask what is the upper tail large deviation rate function, namely, what is the limit of $t^{-1}\log \mathbb{P}(\mathcal{H}^{f}(t,0)+\frac{t}{24}\geq st)$ as $t$ goes to $\infty$.

Our result which we state as follows computes the $p$-th moment Lyapunov exponent of the SHE solutions $\{\mathcal{Z}^{f}\}_{t>0}$ and the upper tail large deviation rate function of the Cole-Hopf solution $\{\mathcal{H}^{f}\}_{t>0}$ when $f$ is deterministic. For any Borel measurable function $f$, let $\mathrm{Osc}_f(I) := sup_{x\in I}f(x) - \inf_{x\in I} f(x)$ be the oscillation of $f$ in the interval $I$.

\begin{theorem}[Deterministic initial data]\label{cor:detinitial}
Let $f$ be a non-random Borel measurable function with sub-exponential growth, i.e., there exist $c_1,c_2>0, \delta\in (0,1)$ such that 
\begin{align}\label{eq:detgrowth}
f(x)\leq c_1+c_2|x|^{\delta}, \quad \forall x\in \mathbb{R}. 
\end{align}
Suppose the local oscillations of $f$ is uniformly bounded, i.e., $\sup_{I\subset \mathbb{R}, |I|\leq \alpha}\mathrm{Osc}_{I}(f)<\infty$ for some $\alpha>0$.  
Then, we have the following:
\\
a) For any $p \in \mathbb{R}_{>0}$, 
\begin{equation}\label{eq:upperbound}
\mathrm{Lya}_p(f) = \frac{p^3 - p}{24}. 
\end{equation} 
b) For all $s \in \mathbb{R}_{>0}$, 
\begin{equation}\label{eq:detLDP}
\lim_{t \to \infty} \frac{1}{t} \log \mathbb{P}\Big(\mathcal{H}^{f} (t, 0) + \frac{t}{24} \,\geq\, st\Big) = -\frac{4 \sqrt{2}}{3} s^{\frac{3}{2}}.
\end{equation} 
\end{theorem}
Theorem~\ref{cor:detinitial} raises the following three important questions. We are indebted to three anonymous referees for urging us to to discuss those questions in details.  
\smallskip
\\
\noindent \emph{Why is \eqref{eq:detLDP} a large deviation result? Why does the speed of the large deviation probability in \eqref{eq:detLDP} equal to $t$?}
\smallskip
\\
\indent It is known that the order of the fluctuation of $\mathcal{H}^{f}(t,0)$ when centered by $-t/24$ is proportional to $t^{1/3}$ as $t$ increases. This result follows from \cite[Theorem 1.1]{ACQ11} when $f$ is the narrow wedge initial data and from \red{\cite[Theorem 1.1 and Theorem 1.10]{CG18b}} when $f$ belongs to a general class of non-random initial data. Note that the size of the deviation in \eqref{eq:detLDP} matches the order of growth of $\mathcal{H}^{f}(t,0)$ which is typical for a large deviation regime. This explains why \eqref{eq:detLDP} is a large deviation result.

Let us explain why the speed of the LDP in \eqref{eq:detLDP} is expected to be $t$. In \cite{CG18b}, the authors showed that the upper tail probability $\mathbb{P}(\mathcal{H}^{f}(t, 0) + \frac{t}{24}>zt^{1/3} )$ is sandwitched between $ e^{-c_1z^{3/2}}$ and $e^{-c_2z^{3/2}}$ for some constants $c_1>c_2>0$ for any finite $z,t>0$. Let us substitute $z=st^{2/3}$ in the above probability. Such substitution marks the large deviation regime as we explained earlier. Under such substitution, the upper and lower bounds on the logarithms of the upper tail probability are proportional to $s^{3/2}t$. This explains why the speed of the LDP is $t$ in \eqref{eq:detLDP}. 
\smallskip
\\
\noindent \emph{How does the sequence $\mathbb{E}[(\mathcal{Z}^{f}(t,x))^p]$ grow with $t$ when $x\neq 0$? What is the upper tail large deviation principle for $\mathcal{H}^{f}(t,x)$ for $x\neq 0$?}
\smallskip
\\
\indent By the convolution principle which is explained in Section~\ref{sec:proofidea}, for any fixed Borel measurable initial data $f$ and $x\in \mathbb{R}$, the law of $\mathcal{Z}^{f}(t,x)$ is same as $\mathcal{Z}^{g}(t,0)$ where $g:\mathbb{R}\to \mathbb{R}$ is defined by $g(y):=f(y - x)$ for $y\in \mathbb{R}$. This being the case, it suffices to study the Lyapunov exponent of the SHE and upper tail large deviation of the KPZ equation at $ x=  0$. Theorem~\ref{cor:detinitial} naturally applies to general $x\in \mathbb{R}$ via shifting the initial data.

\begin{remark}
 Theorem~\ref{cor:detinitial} applies to a large class of initial data  of the SHE. Since the oscillation of any bounded positive initial data is uniformly bounded, Theorem~\ref{cor:detinitial} naturally applies to those initial data. In fact, Theorem~\ref{cor:detinitial} shows the $p$-th Lyapunov exponents of the SHE started from any bounded positive initial data is equal to $p(p^2-p)/24$ for any $p>0$. This gives a positive answer to a question by Xia Chen from \cite[page 1489]{ChenX15}. With mild modification of our proof techniques, one can also show that the Lyapunov exponents of the SHE started from bounded compactly supported initial data are same as that in Theorem~\ref{cor:detinitial}. See Theorem~\ref{thm:main} and Remark~\ref{rem:Compactly} for more details.       
\end{remark}

Our next result deals with the case of random initial data. In particular, we focus on the Brownian initial data for the KPZ equation. The existence and the uniqueness of the solution of the SHE started from the exponential of two-sided Brownian motion with linear drift follows from \eqref{eq:JCcondition}.

\red{\begin{theorem}[Brownian initial data]\label{cor:bminitial}
	Let  $B(x)$ be a two-sided (standard) Brownian motion. For any  $t>0,x\in \mathbb{R}$, define $f(x) := (\sigma_+ B(x) + a_+ x) \mathbbm{1}_{\{x >0 \}} - (-\sigma_- B(x) +  a_- x) \mathbbm{1}_{\{x < 0\}} $  where $\sigma_+, a_+$ and $\sigma_-, a_-$ are the diffusion and drift parameters for $x > 0$ and $x<0$ respectively. 
Then, we have  \\
a) For any $p > 0$, 
	\begin{equation}\label{eq:bmLyapunov}
	\mathrm{Lya}_{p}(f) = \frac{p^3}{24} - \frac{p}{24} +  \frac{p}{2} \max\big\{\frac{p\sigma^2_+}{2} + a_+, \frac{p\sigma^2_-}{2} + a_- ,0\big\}^2.
	\end{equation}
	b) 
	Let $a = \max(a_+, a_-)$.	We have 
	\begin{align}\notag
	&\lim_{t \to \infty} \frac{1}{t} \log \mathbb{P}\Big(\mathcal{H}^{f} (t, 0) + \frac{t}{24} \red{\, \geq \,} st\Big) \\
	\label{eq:bmldp}
	&=\begin{cases} -\max_{p > 0}\Big(sp - \big(\frac{p^3}{24} +  \frac{p}{2} \max\big\{\frac{p\sigma^2_+}{2} + a_+, \frac{p\sigma^2_-}{2} + a_- ,0\big\}^2)\Big) & s\geq \frac{\max(a, 0)^2}{2},
	\\	\red{0} & \red{s\leq \frac{\max(a, 0)^2}{2}.}
	\end{cases}
	\end{align}
The expression of the large deviation rate function is more explicit when $\sigma_+ = \sigma_- = 1$, we have if $a \geq 0$,
\begin{align}\label{eq:bmldp1}
\lim_{t \to \infty} \frac{1}{t} \log \mathbb{P}\Big(\mathcal{H}^{f} (t, 0) + \frac{t}{24} \, \geq \, st\Big) =\begin{cases} 
0 & s\leq \frac{a^2}{2},
\\
- \frac{2\sqrt{2}}{3} s^{\frac{3}{2}} + sa - \frac{a^3}{6} & s\geq \frac{a^2}{2}.
\end{cases}
\end{align}
If $a < 0$, then, 
\begin{equation}\label{eq:bmldp2}
\lim_{t \to \infty} \frac{1}{t} \log \mathbb{P}\Big(\mathcal{H}^{f}(t, 0) + \frac{t}{24} \, \geq\, st\Big) = 
\begin{cases}
0 & s \leq 0,\\
-\frac{4\sqrt{2}}{3}  s^{\frac{3}{2}}  & 0 \leq s \leq \frac{a^2}{2},\\
-\frac{2\sqrt{2}}{3} s^{\frac{3}{2}} + sa - \frac{a^3}{6} & s\geq \frac{a^2}{2} .
\end{cases}
\end{equation}
\end{theorem}
}
 The proofs of Theorem~\ref{cor:detinitial} and~\ref{cor:bminitial} will be given in Section~\ref{sec:Cor}.
\begin{remark}
\red{Set $\sigma_+ = \sigma_- = 1$, then \eqref{eq:bmLyapunov} becomes 
\begin{equation}\label{eq:bmLyapunov1}
\mathrm{Lya}_{p}(f) = \frac{p^3}{24} - \frac{p}{24} +  \frac{p}{2} \max\big\{\frac{p}{2} + a, 0\big\}^2.
\end{equation}}
Theorem~\ref{cor:bminitial} shows a sharp transition of the Lyapunov exponent in \eqref{eq:bmLyapunov1} as the slope parameter $a$ crosses the value $ -\frac{p}{2}$. Note that $\mathbb{E}[e^{pf(x)}] \leq \exp(p(\tfrac{p}{2}+a)|x|)$ for any $x\in \mathbb{R}$ where $f$ is same as in Theorem~\ref{cor:bminitial}. When $a\leq -\frac{p}{2}$, $\mathbb{E}[e^{pf(x)}]$ is either exponentially decaying or, bounded as $|x| $ goes to $\infty$. This effectively ensures (as we show in the proof of Theorem~\ref{cor:bminitial}) that the role of the initial data $e^{f}$ for the SHE is just similar to any bounded positive initial data when it comes to computing $p$-th moment of $\mathcal{Z}^{f}(t,0)$ with $a$ being less than or equal to $-p/2$. Hence, it is not surprising to expect that the $p$-th Lyapunov exponent in \eqref{eq:bmLyapunov} is same as in \eqref{eq:upperbound} in the latter case. On the other hand, $\mathbb{E}[e^{pf(x)}]$ is exponentially growing in $|x|$ when $a$ is greater than $-p/2$. As we will explain in Section~\ref{sec:proofidea}, the exponential growth of $\mathbb{E}[e^{pf(x)}]$ insinuates a tighter dependence of the growth of $p$-th moment on the initial data which finally makes the Lyapunov exponent for the $a>-p/2$ case be rather different than Theorem~\ref{cor:detinitial}. 
\end{remark}

\begin{remark}\label{rem:BmRem}
It is worthwhile to note the contrast between the upper tail large deviation probability (LDP) of the KPZ equation under constant or, narrow wedge initial data and the Brownian initial data. The large deviation rate function in \eqref{eq:bmldp} is $-\frac{2\sqrt{2}}{3}s^{3/2}$ when both the drift parameters $a_{+}$ and $a_{-}$ are equal to $0$. Comparing this rate function with \eqref{eq:detLDP} or, \cite[Theorem 1.1]{DT19} shows a difference by a factor of $2$. This difference between the upper tail LDPs is consistent with the difference of the upper tail asymptotics of the KPZ equation under KPZ scaling. In the physics literature, the contrast between the LDPs under different initial data of the KPZ equation is echoed in \cite{LMS16, LMRS16, janas2016dynamical,MS17, krajenbrink2017exact}.\footnote{In fact, \cite{LMRS16, janas2016dynamical,MS17, krajenbrink2017exact} study the tails of the KPZ equation in the short time regime, yet it is widely believed that the constants in front of the $s^{\frac{3}{2}}$ in the upper tail are the same for short time and long time.} Our result rigorously confirms those predictions. Interestingly, a recent work \cite{ferrari2021upper} has shown that the bounds for the upper tail of the KPZ fixed point (constructed in \cite{MQR16,DOV18}) started from Brownian initial data (resp. any bounded initial data) coincides with the LDP rate function from Theorem~\ref{cor:bminitial} (resp. Theorem~\ref{cor:detinitial}).  
\end{remark}


\subsection{Proof Ideas \& Generalizations}
\label{sec:proofidea}
In this section, we will present the proof ideas of our main results. While this will remain as our main goal throughout this section, to highlight the insights of our techniques, we will also present a generalization of our results for a much broader class of initial data containing those which are covered in Theorem~\ref{cor:detinitial} and~\ref{cor:bminitial}. We first introduce the key tools for our proof. The sequential assembly of these tools which will be sketched thereafter, requires few minimal assumptions on the initial data of the SHE. At the very end of our description of the proof strategy, we will state a general result on the Lyapunov exponents and upper tail LDP of the KPZ equation for the initial data satisfying those minimal assumptions.

We start with setting some necessary notations. The \emph{narrow wedge} solution  $\mathcal{H}^{\mathbf{nw}}$ of the KPZ is the Cole-Hopf transform of the fundamental solution of SHE $\mathcal{Z}^{\mathbf{nw}}$, which is associated to the the delta initial data $\mathcal{Z}^{\mathbf{nw}}(0,x)= \delta_{x=0}$. As $t$ goes to $0$, $\log \mathcal{Z}^{\mathbf{nw}}(t,x)$ is well approximated by the heat kernel whose logarithm is given by a thin parabola $\frac{x^2}{2t}$, rendering $\mathcal{H}^{\mathbf{nw}}$ to have \emph{narrow wedge} like structure.

The proof of our main results consists of following tools: $(1)$ a composition law which connects the SHE under general initial data with its fundamental solution, $(2)$ Lyapunov exponents of the fundamental solution of the SHE, and $(3)$ tails bounds on the spatial regularity of the narrow wedge solution of the KPZ. 

The following identity which is taken from Lemma 1.18 of \cite{CH16} gives a convolution formula of the one point distribution of the KPZ equation in terms of the spatial process $\mathcal{Z}^{\mathbf{nw}}(t,\cdot)$ and the initial data $f$ of the KPZ equation. The proof of this formula is associated to the linearity and time reversal property of the SHE. For reader's convenience, we provide a proof of this identity in Section~\ref{sec:Aux3}. 

\begin{prop}[Convolution Formula]\label{prop:convolution}
Let $\mathcal{Z}^{f}$ be the unique solution of the SHE started from the initial condition $e^{f}$ such that $\int_{A} e^{f(x)} dx$ is a Radon measure in $A$. Suppose $f$ is independent of the white noise $\xi$ of \eqref{eq:shegeneral} when $f$ is random. Then for any $t>0$, 
\begin{equation}\label{eq:conv}
\mathcal{Z}^f (t, 0) \overset{d}{=} \int \mathcal{Z}^{\mathbf{nw}}(t, y) e^{f (y)} dy.
\end{equation}
\end{prop}

To complement Proposition~\ref{prop:convolution}, we will make use of the exact expressions of any real positive moment Lyapunov exponents of the fundamental solution of the SHE from \cite{DT19} and the tail bounds on the spatial regularity of the narrow wedge solution of the KPZ equation from \cite{CGH19}. The following result describes the first of these two tools.  
  \begin{prop}[Lyapunov exponents of fundamental solution, \cite{DT19}, Theorem 1.1]\label{prop:DT19thm}
For every $p > 0$, 
\begin{equation*}
\lim_{t \to \infty} \frac{1}{t} \log \mathbb{E}\Big[\mathcal{Z}^{\mathbf{nw}}(t, 0)^p\Big] = \frac{p^3-p}{24}.
\end{equation*}
\end{prop}  

The third of our main tools is (super)-exponential tail bounds on the spatial regularity of the narrow wedge solution of the KPZ equation. This is given in the following result.

\begin{prop}[Tail bounds of increments, \cite{CGH19}, Prop. 4.4] \label{prop: CGH 4.4} 
For any $t_0 > 1$, $\nu > 0$ and $\epsilon\in (0,1)$, there exist $s_0 = s_0(t_0, \nu,\epsilon)$ and $c = c(t_0, \nu,\epsilon)$ such that, for $t \geq t_0$ and $s \geq s_0$,
\begin{align*}
\numberthis
\label{eq:increment upper tail}
\mathbb{P}\Big(\sup_{x \in [0, t^{\frac{1}{3}}]} \Big\{\mathcal{H}^{\mathbf{nw}} (t,x) - \mathcal{H}^{\mathbf{nw}} (t,0) - \frac{\nu x^2}{2}\Big\} \geq s\Big) \leq \exp(-c s^{\frac{9}{8}-\epsilon}),
\\
\numberthis
\label{eq:increment lower tail}
\mathbb{P}\Big(\inf_{x \in [0, t^{\frac{1}{3}}]} \Big\{\mathcal{H}^{\mathbf{nw}} (t,x) - \mathcal{H}^{\mathbf{nw}} (t,0) + \frac{\nu x^2}{2}\Big\} \leq - s\Big) \leq  \exp(-c s^{\frac{9}{8}-\epsilon}).
\end{align*}
\end{prop}

In order to prove \eqref{eq:upperbound} (resp. \eqref{eq:bmLyapunov}) of Theorem~\ref{cor:detinitial} (resp. Theorem~\ref{cor:bminitial}), we first use Proposition~\ref{prop:convolution} to note 
\[\mathrm{Lya}_{p}(f) = \lim_{t\to \infty}\frac{1}{t}\log\Big(\mathbb{E}\Big[\Big(\int_{\mathbb{R}} \mathcal{Z}^{\mathbf{nw}}(t,y)e^{f(y)}dy\Big)^{p}\Big]\Big)\]
and thereafter, focus our effort to analyze $\mathbb{E}[(\int_{\mathbb{R}} \mathcal{Z}^{\mathbf{nw}}(t,y)\exp(f(y))dy)^{p}]$.
Our main technical achievement is to justify the following heuristic approximation 
\begin{align}\label{eq:Hue}
\frac{1}{t}\log \mathbb{E}\Big[\Big(\int_{\mathbb{R}} \mathcal{Z}^{\mathbf{nw}}(t,y)e^{f(y)}dy\Big)^{p}\Big] \approx \frac{1}{t}\log \int_{\mathbb{R}}\mathbb{E}\big[(\mathcal{Z}^{\mathbf{nw}}(t,y))^{p}\big]\mathbb{E}[e^{p f(y)}] dy,
\end{align}
for all large $t\in \mathbb{R}_{>0}$ where $``\approx"$ indicates the equality up to some additive constant which decays as $t\to \infty$. For showing this, we will observe that the main contributions of the left hand side of \eqref{eq:Hue} comes from $\mathcal{Z}^{\mathbf{nw}}(t,y_0)e^{f(y_0)}$ where $y_0$ is a point in $\mathbb{R}$ such that the function $\phi(y):=-\frac{py^2}{2}+\log \mathbb{E}[e^{pf(y)}]$ when evaluated at $y_0$ attains a close proximity to the supremum value $\sup_{y \in R}\phi(y)$. Similarly, we find that the main contribution of the right hand side of \eqref{eq:Hue} comes from $\mathbb{E}\big[(\mathcal{Z}^{\mathbf{nw}}(t,y_0))^{p}\big]\mathbb{E}[e^{pf(y_0)}]$. Note that the integrand in the right hand side of \eqref{eq:Hue} is a product of expectations. One of the main reason for expecting such product is the independence between the time-space white noise $\xi$ and the initial data $f$. 

For showing there are indeed such local representatives of both sides of \eqref{eq:Hue}, we require to demonstrate that the contributions of $\mathbb{E}\big[(\int_{\mathbb{R}\backslash B}\mathcal{Z}^{\mathbf{nw}}(t,y)e^{f(y)}dy)^p\big]$ and $\int_{\mathbb{R}\backslash B}\mathbb{E}\big[(\mathcal{Z}^{\mathbf{nw}}(t,y))^{p}\big]\mathbb{E}[e^{pf(y)}]dy$ cannot grow significantly higher than their local counterparts where $B$ is small interval around $y_0$. This is done by controlling fluctuation of the spatial process $\mathcal{Z}^{\mathbf{nw}}(t,\cdot)$ and  the growth and regularity of the initial data $f$. The fluctuation of $\mathcal{Z}^{\mathbf{nw}}(t,\cdot)$ is controlled by the tail probability bounds \eqref{eq:increment upper tail} and \eqref{eq:increment lower tail} on the spatial regularity of $\mathcal{H}^{\mathbf{nw}}(t,\cdot)$. The growth and regularity estimates of $f$ in Theorem~\ref{cor:detinitial} and~\ref{cor:bminitial} will be derived separately. However, one may wonder what are the minimal conditions on the growth and regularity of $f$ under which the current proof techniques work. We list those minimal sufficient conditions below after introducing few more notations.

To extend our result to the full generality, we also allow the initial data to depend on time $t$. We use the notation $f_t$ to denote time dependent initial data. We denote the SHE solutions started from $e^{f_t}$ by $\mathcal{Z}^{f_t}$. More precisely, for every fixed $t \geq 0$, $\mathcal{Z}^{f_t} (t, \cdot)$ equals the time-$t$ solution to the SHE with initial data $e^{f_t}$. As usual, we denote the corresponding Cole-Hopf solution of the KPZ equation by $\mathcal{H}^{f_t}:= \log \mathcal{Z}^{f_t}$.

We want to mention that the time dependent initial data is not unusual. Due to the connection with the KPZ universality class, it is natural to consider the initial data which varies in the $(t^{1/3}, t^{2/3})$-scale. For instance, \cite[Definition~1.4]{CH16} and  \cite[Definition~1.1]{CG18b} had considered $f_t$ to be a relevant initial data for the KPZ equation if $t^{-1/3}f_t(t^{2/3}x)$ is upper bounded by a parabola with leading constant less than $\frac{1}{2}$ and lower bounded by a constant in a compact interval around $0$. It is straightforward to check that when $f_t(x)$ is equal to $\alpha \frac{x^2}{2t}$, both of the above conditions are satisfied.

Now we state our conditions on $f_t$. For $x\in \mathbb{R}$ and $t\in \mathbb{R}_{\geq 0}$, define 
\begin{align}
M^{f_t}_{p}(t,x) := \begin{cases}
\mathbb{E}[e^{p f_t(x)}] & \text{when }f_t(x) \text{ is random},\\  e^{p f_t(x)} & \text{when }f_t(x) \text{ is deterministic}.
\end{cases}
\end{align}

\begin{enumerate}
\item[(1)] (Growth and lower bound conditions:) For each $p\in \mathbb{R}_{>0}$, there exist constants $C, K, L>0$ and $0 < \alpha < 1$ depending on $p$ such that for all $t> 0$,
\begin{align}\label{eq:growth condition}
M^{f_t}_p(t,x) \leq C (e^{C |x|} + e^{\frac{\alpha p x^2}{2t}}),\\
\label{eq:lower bound condition}
\sup_{x \in [-L, L]} \log M^{f_t}_p(t,x) >-K.
\end{align}

\item[(2)] (Pseudo-stationarity:) We call $\{\theta_{n}\}_{n\in \mathbb{Z}}\subset \mathbb{R}$ a sequence of grid points if it satisfies $$\ldots <\theta_{-1}<\theta_0=0<\theta_1<\ldots, \quad \lim_{n\to \infty}\theta_n=\infty, \quad \lim_{n\to -\infty}\theta_{n}=-\infty, \quad (\max\{c|n|,1\})^{-\beta}\leq |\theta_{n}-\theta_{n+1}|\leq 1$$ for some $c>0$, $\beta\in (0,1)$ and all $n\in \mathbb{Z}$. There exist constants $C, t_0,s_0>0$ and a sequence of grid points $\{\theta_{n}\}_{n\in \mathbb{Z}}$  such that for all $t>t_0$, $n\in \mathbb{Z}$, $\sup_{x\in [\theta_{n},\theta_{n+1}]}|f_t(x)- f_t(\theta_{n})|\leq s_0$ when $f_t$ is deterministic or, 
\begin{align}\label{eq:ModCont}
\mathbb{P}\Big(\sup_{x\in [\theta_{n},\theta_{n+1}]}|f_t(x)- f_t(\theta_{n})|\geq s\Big)\leq e^{-Cs^{1+\delta}} ,\qquad \forall s\geq s_0,
\end{align} 
when $f_t$ is random.
\end{enumerate}

The \emph{growth and lower bound condition} controls the pointwise exponential moments of the initial data. If the initial data is time independent, then, \eqref{eq:JCcondition} provides the sufficient condition for the existence and the uniqueness of the SHE. For the time dependent initial data $f_t$, the existence and the solution of the SHE till time $t$ is not explicitly written down anywhere. However, using the techniques in \cite{CD15} which is based on the chaos expansion of the mild solution of the SHE, it is straightforward to check that the SHE starting from $e^{f_t}$ has unique solution up to time $t$ if the following condition is satisfied:
\begin{align}\label{eq:CDcondition}
\int e^{-\frac{x^2}{2t}} \mathbb{E}[e^{f_t(x)}] dx<\infty.
\end{align} 
Not only the growth and lower bound condition ensures the existence and uniqueness solution of the stochastic heat equation up to desired time, but also helps to provide many error estimates in the proof of Theorem~\ref{thm:main}.

The \emph{pseudo-stationarity} condition provides an uniform upper bound to the maximal variation of the initial data one a sequence of intervals. Since the initial data can satisfy this property without being spatially stationary, we call this  pseudo-stationarity condition.
 
One of the easy examples where taking constant grid in the pseudo-stationarity is not enough is the time dependent parabolic initial data, i.e., $f_t(x) := \alpha x^2/2t$. Since $f_t$ satisfies \eqref{eq:CDcondition}, the solution of the Cole-Hopf solution of the KPZ started from $f_t$ exists up to time $t$. The oscillations of $f_t(x)$ over any set of grid points of constant spacing will be increasing as we increase $x$ to $\infty$. However, if we define $\theta_n$ to be $\mathrm{sign}(n)\times |n|^{1/2}$, then the oscillation of $f_t(x)$ over $[\theta_n,\theta_{n+1}]$ is bounded by $1/2t$ which does not depend $n$. In this example, it is mandatory to take grid points of different spacing.


Let us again come back to analyzing the right hand side of \eqref{eq:Hue}. To analyze the integral on the right hand side of \eqref{eq:Hue}, one may first ask how we deal with $\mathbb{E}\big[(\mathcal{Z}^{\mathbf{nw}}(t,y))^{p}\big]$ for all $y\in \mathbb{R}$. This will be done by combining Proposition~\ref{prop:DT19thm} with the following result.

\begin{prop}[Stationarity, \cite{ACQ11}, Prop. 1.4]\label{prop:stationarity}
For any fixed $t > 0$, the random process $\mathcal{H}^{\textbf{nw}} (t, x) + \frac{x^2}{2t}$ is stationary in $x$. 
\end{prop}

 Proposition~\ref{prop:stationarity} allows us to write the right hand side of \eqref{eq:Hue} as a sum of two terms, namely, $t^{-1}\log(\mathbb{E}[(\mathcal{Z}^{\mathbf{nw}}(t,0))^{p}])$ and $t^{-1}\log\big(\int_{\mathbb{R}}\exp(-\frac{px^2}{2t})\mathbb{E}[e^{p f_t (x)}]dx\big)$. To compute the $t \to \infty$ limit of the right hand side of \eqref{eq:Hue}, it suffices to compute the $t\to \infty$ limits of these two terms
 respectively.
 By Proposition~\ref{prop:DT19thm}, the first limit equals $\frac{p^3 - p}{24}$. Our proof will show that the second limit equals $g(p)$ (defined in \eqref{eq:Gp}), which is guaranteed by our next condition.
 
 \begin{enumerate}
 \item[(3)] (Coherence conditions:) 
 The following limit exists
 \begin{align}
\lim_{t \to \infty} \frac{1}{t} \sup_{x \in \mathbb{R}} \Big\{\frac{-p x^2}{2t} + \log M^{f_t}_p(x,t)\Big\}=:g(p)\label{eq:Gp}
\end{align}
for all $p\in \mathbb{R}_{>0}$.
Furthermore, for every $p\in \mathbb{R}_{>0}$,
\begin{equation}\label{eq:condition 3}
\liminf_{\epsilon \to 0} \limsup_{t \to \infty} \frac{1}{t} \log \bigg(\int e^{-\frac{p(1-\epsilon) x^2}{2t}}M^{f_t}_{p(1+\epsilon)}(t,x) dx\bigg) \leq g(p). 
\end{equation}
 \end{enumerate}
The \emph{coherence condition} captures the contribution of the initial data in the expression of the Lyapunov exponents via solution of a variational problem. On the one hand, it enforces the existence of the solution of that variation problem  and on the other hand, it imparts sufficient convexity to the initial data so that the solution of that variational problem stay inside large bounded domain.

Next we state a more general result on Lyapunov exponents for the class of initial data satisfying the above conditions.
\begin{defin}\label{Def:HypDef}
We call a set of measurable functions $(g,\{f_{t}\}_{t> 0})$ with deterministic $g:\mathbb{R}_{>0} \to \mathbb{R}_{\geq 0}$ and $f_{t}:\mathbb{R}\to \mathbb{R}$ (possibly random) belongs to the class $\mathbf{Hyp}$ if $\{f_{t}\}_{t\geq 0}$ satisfies (1) growth and lower bound conditions, (2) pseudo-stationarity condition and (3) coherence conditions. 
\end{defin}
For any $(g,\{f_t\}_{t> 0})\in \mathbf{Hyp}$, there exists a unique solution of the SHE started from initial data $e^{f_t}$ up to time $t$ thanks to the growth and lower bound condition, see \eqref{eq:CDcondition}.

For any $p\in \mathbb{R}_{>0}$, we define the $p$-th moment Lyapunov exponent of the class $\{f_t\}_{t>0}$ as 
\begin{align}
\mathrm{Lya}_{p}(\{f_t\}_{>0}) := \lim_{t\to \infty} \frac{1}{t}\log \mathbb{E}\Big[\big(\mathcal{Z}^{f_t}(t,0)\big)^p\Big].
\end{align}
Now we are ready to state the more general result on the Lyapunov exponents. 
\begin{theorem}\label{thm:main}
 Let $\big(g,\{f_t\}_{t\geq0}\big)$ be a set of functions in the class $\mathbf{Hyp}$. Then, we have the following:
 \begin{enumerate}
\item[(a)] For any $p\in \mathbb{R}_{>0}$, \begin{align}\label{eq:intermittency}
\mathrm{Lya}_{p}(\{f_t\}_{t>0}) = \frac{p^3-p}{24} +g(p) .
\end{align} 
\item[(b)] Suppose $g(p)\in C^{1}(\mathbb{R}_{>0})$ and $ \zeta:=\lim_{p\to 0}g^{\prime}(p)$ is finite. Then, for $s > \zeta$, 
\begin{align}\label{eq:ldp}
\lim_{t\to \infty} \frac{1}{t} \log \mathbb{P}\Big(\mathcal{H}^{f_t}(t,0)+ \frac{t}{24} \, \geq\,   st\Big) = -\max_{p\geq 0} \big\{sp -\frac{p^3}{24}-g(p) \big\}. 
\end{align} 
\end{enumerate}
\end{theorem} 
As we have pointed out before, our techniques for proving Theorem~\ref{cor:detinitial} and~\ref{cor:bminitial} are flexible enough to extend to the initial data coming from the class $\mathbf{Hyp}$. In fact, we first prove Theorem~\ref{thm:main} in Section~\ref{sec:MainTheorem} and then, prove Theorem~\ref{cor:detinitial} and~\ref{cor:bminitial} as its corollary. For instance, Theorem~\ref{cor:detinitial} will be proved using Theorem~\ref{thm:main} by defining $g(p):=0$ and $f_t:= f$ (with $f$ satisfying the condition in Theorem \ref{cor:detinitial}) for all $t>0$ and finally, showing $(g,\{f_t\}_{t\geq 0})$ belongs to the class $\mathbf{Hyp}$.
\begin{remark}\label{rem:Compactly}
It is straightforward to check that if $f_t = f$ for all $t$ such that $e^{f}$ is strictly positive and bounded on a compact interval, then $\{f_t\}_{t>0}$ satisfies the growth and lower bound conditions (\eqref{eq:growth condition} and \eqref{eq:lower bound condition}) and the coherence condition (with $g = 0$). However, one needs minor modification of the pseudo-stationarity condition for it to be valid when $e^{f}$ is compactly supported. After doing the appropriate modifications, one can use similar proof ideas to extend the results of Theorem~\ref{thm:main} in the case of compactly supported initial data for the SHE.
\end{remark}
We have explained the ideas of obtaining \eqref{eq:intermittency}. 
However, it still remains to explain how \eqref{eq:ldp} follows from \eqref{eq:intermittency}. 
For showing \eqref{eq:ldp} of Theorem~\ref{thm:main}, our main tool is the following proposition which relates the upper tail large deviation rate function of $\mathcal{H}^{f_t}$ in terms of the Lyapunov exponents. 

\begin{prop}\label{prop:intermittldp}
Let $X(t)$ be a stochastic process indexed by $t\in \mathbb{R}_{>0}$. Fix $h \in C^{1}(\mathbb{R}_{> 0})$ such that $h': (0, \infty) \to (\zeta, \infty)$ is continuous, bijective and increasing for some $\zeta \in \mathbb{R}$. Assume that
\begin{equation}\label{eq:temp1}
\lim_{t \to \infty} \frac{1}{t} \log \mathbb{E}\Big[e^{p X(t)}\Big] = h(p), \quad \forall p\in \mathbb{R}_{>0}.
\end{equation}
Then, we have 
\begin{equation}\label{eq:ldp1}
\lim_{t \to \infty} \frac{1}{t} \log \mathbb{P}\Big(X(t) \, \geq \, st\Big) = - \sup_{p > 0}\{ps - h(p)\}, \quad \forall s>\zeta. 
\end{equation}
\end{prop}
We believe that this proposition should be contained in some previous works since it bears a lot of similarities with the Gartner-Ellis theorem (see \cite[Section~2.3]{DZ10}). We provide a proof in Section~\ref{sec:Aux1} since we are unable to find the exact statement in the literature. 
Allying Proposition~\ref{prop:intermittldp} with \eqref{eq:intermittency} yields the proof of \eqref{eq:ldp}. It is worthwhile to note that the use of Proposition~\ref{prop:intermittldp} necessitates $g(\cdot)$ (in \eqref{eq:Gp}) to be a convex function and to satisfy few other technical properties. Under the assumption that $\big(g,\{f_t\}_{t\geq 0}\big)$ belong to a class $\mathbf{Hyp}$, the following lemma shows that $g(\cdot)$ indeed satisfies those properties. The proof of this lemma is deferred to Section~\ref{sec:Aux2}.

\begin{lemma}\label{lem:condition}
For any set of functions $\big(g,\{f_t\}_{t\geq 0}\big)$ in the class $\mathbf{Hyp}$, we have the following: 
\begin{enumerate}[(i)]
\item $g$ is convex and non-negative.
\item For every $p > 0$ and $\gamma>0$, define
\begin{equation}
\mathrm{MAX}^{f}_{p,\gamma}(t):=\Big\{x: -\frac{p x^2}{2t} + \log M^{f_t}_p(t,x)\geq \sup_{y\in \mathbb{R}}\big\{-\frac{p y^2}{2t} + \log M^{f_t}_p(t,y)\big\}-\gamma\Big\}.\label{eq:MAXDef}
\end{equation}
Then there exists $T_0 > 0$ such that for $t > T_0$, $\mathrm{MAX}^{f}_{p,\gamma}(t)$ is nonempty for all $p,\gamma>0$. Define $x_{p,\gamma}(t):= \argmax\big\{|x|:x\in \mathrm{MAX}^{f}_{p,\gamma}(t)\big\}$. There exists a constant $C=C(p,\gamma)>0$ such that for all $t > T_0$, $|x_{q,\gamma}(t)| \leq C t$ for all $\frac{p}{2} < q < 2p$. 
\end{enumerate}

\end{lemma}

\begin{remark}
The class $\mathbf{Hyp}$ in Theorem~\ref{thm:main} contains a large collection of interesting initial profiles for the KPZ equation. It is only bounded deterministic initial data and the delta initial data of the SHE for which all integer moment Lyapunov exponents were known (see \cite{BC95, BC15, ChenX15, CG18b}) before. We would like to stress that the \emph{narrow wedge initial data} of the KPZ equation which corresponds to taking $\mathcal{Z}(0,x)= \delta_{x=0}$ (i.e., the delta initial data of the SHE) is not covered by Theorem~\ref{thm:main}. However, the real positive moment Lyapunov exponents in the narrow wedge case are recently found in \cite{DT19} and those are one of the key inputs to our proof of Theorem~\ref{thm:main}.    
\end{remark}

\subsection{Previous works}
Our main result on the Lyapunov exponent of the SHE and the upper tail large deviation of the KPZ equation fits into the broader endeavor of studying the intermittency phenomenon and large deviation problems of the random field solution of stochastic partial differential equations. Intermittency, an universal phenomenon for random fields of multiplicative type is characterized by enormous moment growth rate of the random field. The nature of the intermittency is captured through the Lyapunov exponents. In last few decades, there were extensive amount of works on studying the growth rate of Lyapunov exponents  under variation in structure of the noise \cite{GM90,CM94,BC95, FK09,CJKS, HHNT, CD15,BC16} and the partial differential operators \cite{Ch17,CHN19}. Large deviation of the stochastic partial differential equations \cite{HW15, CD19} is an active area of research in recent years. Upper and lower tail large deviation of the KPZ equation behold special interests in theoretical as well as in experimental side and have been recently investigated in a vast amount of works. For detailed history along this line of works, we refer to \cite{LMRS16,HLMRS18,CGKLT,T18,DT19,KL19} and the references therein. Below, we compare our results with few of those previous works.

Based on the replica Bethe ansatz techniques, Kardar \cite[Section 2.2]{Kardar87} predicted the integer moment Lyapunov exponents of the fundamental solution of the SHE. Bertini and Cancrini \cite[Section~2.4]{BC95} made a rigorous attempt to show the exact match between the integer moment Lyapunov exponents of the SHE under constant initial data and Kardar's prediction. Unfortunately,  the computation of \cite{BC95} was incorrect beyond the second moment Lyapunov exponent. This was later fixed by \cite{ChenX15} who computed all integer moments Lyapunov exponents for any deterministic bounded positive initial data of the SHE. The main tool of \cite{ChenX15} was the  moment formulas of the SHE in terms of integral of local time of Brownian bridges derived from the Feynman-Kac representation of the solution.

Alternatively, the integer moments of the fundamental solution of the SHE which are widely believed to be same as the solution of the attractive delta-Bose gas have formulas in terms of contour integrals. We refer to \cite{G18} and the reference therein for a comprehensive discussion on this. Similar formulas are known for the moments of the parabolic Anderson model, semi-discrete directed polymers, $q$-Whittaker process (see \cite{BC15,BC14}) etc. By analyzing the contour integrals, \cite{CG18b} derived a sharp upper and lower bound to the integer moments of the fundamental solution of SHE which positively confirms Kardar's prediction. Recently, \cite{DT19} were able to obtain similar tight upper and lower bound to the fractional moments. Using sharp bounds on the moments, \cite{DT19} computed any positive fractional moment Lyapunov exponent of the fundamental solution. As an application of their result, \cite{DT19} also found the one point upper tail large deviation of the narrow wedge solution of the KPZ equation. We refer to \cite{HHNT,CHKN} for tight bounds on the moments of the SHE when the noise is colored in space/time and the initial data is a continuous bounded function.

In spite of a substantial amount of works on the fractional moments Lyapunov exponents of the SHE with colored noise, the case of general initial data for the SHE with white noise was largely being untouched. The same conclusion applies to the status of the upper tail large deviation result for the KPZ equation started from general initial data. However, tight bounds on the upper tail probabilities of the KPZ are available. For instance, \cite{CG18b} obtained the following result: for any $t_0>0$, there exists $s_0=s_0(t_0)$, $c_1=c_1(t_0), c_2=c_2(t_0)>0$ such that for all $s>s_0$ and $t>t_0$,
\[
e^{-c_1s^{3/2}}\leq \mathbb{P}\big(\mathcal{H}(t,0)+\frac{t}{24}\geq st^{1/3}\big)\leq e^{-c_2s^{3/2}},\] 
where the initial data of the KPZ solution $\mathcal{H}$ belongs to a large class of functions including the narrow wedge and the stationary initial data. We refer to Section~ of \cite{CG18b} and the references therein for more information. In the physics literature, the upper tail large deviation of the KPZ equation has been studied recently using optimal fluctuation theory which corresponds to Freidlin-Wentzell type large deviation theory of stochastic PDEs with small noise. By formal computations, \cite{MKV16, JKM16, MS17} (see also \cite{LMRS16,LMS16}) demonstrated the upper tail LDP of the KPZ started from a large class of initial data including the flat and stationary data. Theorem~\ref{cor:detinitial} and~\ref{cor:bminitial} rigorously confirms those results from physics literature. In a way, Theorem~\ref{thm:main} is the first result which  provides a concrete pathway to compute the Lyapunov exponent of the SHE started from general initial data and the upper tail large deviation rate function of the associated Cole-Hopf solution of the KPZ equation.

The probability of the KPZ equation being smaller than its typical value is captured through its lower tail probability. Like the upper tail, one point lower tail probabilities of the KPZ equation are equally important. The first tight estimates of the lower tail probabilities of the narrow wedge solution is obtained in \cite{CG18a} and the lower tail large deviation is rigorously proved in \cite{T18,CC19} (see also \cite{CGKLT,KL19} and the reference therein). The case of general initial data was considered in \cite{CG18b} where the authors provided an upper bound to the lower tail probability of the KPZ equation.  However, there are only very few things known about the lower tail large deviation under general initial data. 
In the physics literature, recently \cite{LD19} found a connection between the latter and the Kadomtsev-Petviashvili (KP) equation. It is unclear to us how much of the techniques of the present paper will help to get the lower tail large deviation of the KPZ equation started from general initial data.

Recently, there is a growing interest in studying the KPZ equation on the half-line. The large deviations of the KPZ equation on the half-line was studied in \cite{krajenbrink2018large, meerson2018large, T18, lin2020lyapunov}. For the half-line SHE with Robin boundary condition, 
under the narrow wedge initial data, 
\cite{lin2020lyapunov} computes the real positive Lyapunov exponents. As a consequence, the author obtains the upper tail large deviations of the half-line KPZ equation. It is an interesting question to see whether our method allows to obtain the Lyapunov exponents the half-line SHE with Robin boundary condition and general initial condition. One stumbling block could be that the convolution formula in Proposition \ref{prop:convolution} may not hold in the half-line situation, since the time inversion property in the proof of Proposition \ref{prop:convolution} (see Section \ref{sec:Aux3}) fails due to  the half-line boundary.
\subsection*{Outline} Section~\ref{sec:MainTheorem} will prove the Theorem~\ref{thm:main}. Applying Theorem~\ref{thm:main}, Theorem~\ref{cor:detinitial} and~\ref{cor:bminitial} will be shown in Subsections~\ref{sec:cordet} and \ref{sec:corbm} of Section~\ref{sec:Cor}. Proofs of Proposition~\ref{prop:intermittldp}, Lemma~\ref{lem:condition} and Proposition \ref{prop:convolution} are given in Subsections~\ref{sec:Aux1}, \ref{sec:Aux2} and \ref{sec:Aux3} of Section~\ref{sec:Aux}. 







\smallskip

\noindent \emph{Acknowledgment:} PG and YL would like to thank Ivan Corwin, Sayan Das, Shalin Parekh for helpful conversations, and three anonymous referees for helpful comments. YL was partially supported by the Fernholz Foundation's ``Summer Minerva Fellow" program and also received summer support from Ivan Corwin's NSF grant DMS-1811143, DMS-1664650.

\section{Lyapunov exponents and large deviation: Proof of Theorem~\ref{thm:main}}\label{sec:MainTheorem}
The main goal of this section is to prove Theorem~\ref{thm:main}. The part (a) of Theorem~\ref{thm:main} is to compute the Lyapunov exponents $\lim_{t\to \infty}t^{-1} \log\mathbb{E}[\mathcal{Z}^{f_t}(t,0)^p]$ for all $p\in \mathbb{R}_{>0}$. The part (b) involves showing the upper tail large deviation rate function of the KPZ equation. Both of these two results are proved for general initial data. The part (b) is a straightforward consequence of part (a). This is shown in Section~\ref{sec:ProofLDP} using Proposition~\ref{prop:intermittldp}. We prove part (a) as follows.

Note that \eqref{eq:intermittency} follows once we show for all $p\in \mathbb{R}_{>0}$,
\begin{align}\label{eq:Limsupinf}
 \underbrace{\frac{p^3 - p}{24} + g(p)\leq \liminf_{t\to \infty}t^{-1} \log\mathbb{E}[\mathcal{Z}^{f_t}(t,0)^p]}_{\mathsf{LimInf}_p}\leq \underbrace{\limsup_{t\to \infty}t^{-1} \log\mathbb{E}[\mathcal{Z}^{f_t}(t,0)^p]\leq \frac{p^3 - p}{24} + g(p)}_{\mathsf{LimSup}_p}.
\end{align}
We denote the left and right inequality by $\mathsf{LimInf}_p$ and $\mathsf{LimSup}_p$ and the proof of these inequalities will be shown in Section~\ref{sec:Upper bound} and~\ref{sec:Lower bound} respectively.


%

\subsection{Proof of $\mathsf{LimSup}_p$ for all $p\in \mathbb{R}_{>0}$}\label{sec:Upper bound} 

We divide the proof in two stages. In \emph{Stage 1}, we prove $\mathsf{LimSup}_p$ inequality when $p >1$ and \emph{Stage 2} will cover the case when $p\in (0,1]$.
\smallskip

\subsubsection{\emph{Stage 1:}} There are two main steps in the proof of this stage. The first step is to obtain the following upper bound 
\begin{align}\label{eq:LimSup1st}
\mathbb{E}\Big[\mathcal{Z}^{f_t}(t, 0)^p\Big]\leq  \Big(\frac{2\pi t}{\epsilon q}\Big)^{\frac{p}{2q}} \mathbb{E}\Big[\int_{-\infty}^{\infty} e^{\frac{\epsilon p x^2}{2t}}\mathcal{Z}^{\mathbf{nw}} (t, x)^p  e^{p f_t (x)} dx\Big]
\end{align}
for $q = \frac{p}{p-1}$ and arbitrary $\e > 0$. The method is to apply H\"older's inequality in the convolution formula of Proposition~\ref{prop:convolution}. The second step is to bound the expectation of the right hand side of the above display. For this, we first distribute the expectation over $\mathcal{Z}^{\mathbf{nw}} (t, x)^p$ and $e^{p f_t (x)}$ as $x$ varies in $\mathbb{R}$. The computation of the expectation of $\mathcal{Z}^{\mathbf{nw}} (t, x)^p$ for $x\in \mathbb{R}$ will be carried out using the spatial stationarity of $\mathcal{Z}^{\mathbf{nw}} (t, x)$ from Proposition~\ref{prop:stationarity} and the narrow wedge LDP from Proposition~\ref{prop:DT19thm}. For the upper bound on the part involving $e^{p f_t (x)}$, we use the property \eqref{eq:condition 3}. Below, we give details of each step.

By the convolution formula of Proposition~\ref{prop:convolution}, $\mathbb{E}[(\mathcal{Z}^{f_t} (t, 0))^p]$ is equal to $\mathbb{E}[(\int_{-\infty}^{\infty} \mathcal{Z}^{\mathbf{nw}} (t, x) e^{f_t (x)} dx)^p]$. In what follows, we bound $\int_{-\infty}^{\infty} \mathcal{Z}^{\mathbf{nw}} (t, x) e^{f_t (x)} dx$ in order to show \eqref{eq:LimSup1st}. 
Denote by $q = \frac{p}{p-1}$. We write $ \mathcal{Z}^{\mathbf{nw}} (t, x) e^{f_t (x)}$ as a product of $ e^{-\frac{\epsilon x^2}{2t}} $ and $e^{\frac{\epsilon x^2}{2t}} \mathcal{Z}^{\mathbf{nw}} (t, x) e^{f_t (x)}$. By applying H\"{o}lder's inequality
\begin{align*}
\int_{-\infty}^{\infty} \mathcal{Z}^{\mathbf{nw}} (t, x) e^{f_t (x)} dx & \leq \Big(\int_{-\infty}^{\infty} e^{-\frac{\epsilon q x^2}{2t}} dx\Big)^{\frac{1}{q}} \Big( \int_{-\infty}^{\infty} e^{\frac{\epsilon p x^2}{2t}}\mathcal{Z}^{\mathbf{nw}} (t, x)^p e^{p f_t (x)} dx \Big)^{\frac{1}{p}},
\end{align*}
The last inequality in conjunction with the fact that $\int_{-\infty}^{\infty} e^{-\frac{\epsilon q x^2}{2t}} dx$ is equal to $\sqrt{2\pi t/\epsilon q}$ yields 
\begin{align*}
\label{eq:temp5} \mathbb{E}\bigg[\Big(\int_{-\infty}^{\infty} \mathcal{Z}^{\mathbf{nw}} (t, x) e^{f_t (x)} dx\Big)^p\bigg]\leq  \Big(\frac{2\pi t}{\epsilon q}\Big)^{\frac{p}{2q}} \mathbb{E}\Big[\int_{-\infty}^{\infty} e^{\frac{\epsilon p x^2}{2t}}\mathcal{Z}^{\mathbf{nw}} (t, x)^p  e^{p f_t (x)} dx\Big].
\end{align*}
Note that the above inequality shows the upper bound in \eqref{eq:LimSup1st}. We apply Fubini's theorem to interchange the expectation and the integral in the above display. Using the stationarity of $\mathcal{Z}^{\mathbf{nw}} (t, x)e^{x^2/2t}$ (see Proposition \ref{prop:stationarity}), one can write the expectation in the right hand side of the above display as the product of $\mathbb{E}\big[\mathcal{Z}^{\mathbf{nw}}(t, 0)^p\big] $ and $\int_{-\infty}^{\infty} e^{\frac{-(1-\epsilon) p x^2}{2t}} M^{f_t}_p(t,x) dx$ where $M^{f_t}_p(t,x)$ is defined in the coherence conditions (see Definition~\ref{Def:HypDef}) for the KPZ data $(g,\{f_t\}_{t>0})$ . Taking logarithm on both sides of the inequality, dividing by $t$ and letting $t\to \infty$, $\epsilon\to 0$ shows 
\begin{align*}
\limsup_{t\to \infty}\frac{1}{t}\log \mathbb{E}\Big[(\mathcal{Z}^{f_t}(t, 0))^p\Big]\leq \frac{p^3-p}{24} + \liminf_{\epsilon \to 0}\limsup_{t \to \infty} \frac{1}{t} \log\Big(\int_{-\infty}^{\infty}  e^{-\frac{(1-\epsilon) p x^2}{2t}} M^{f_t}_p(t,x) dx\Big),
\end{align*}
where the factor $(p^3-p)/24$ in the right hand side is obtained by applying Proposition \ref{prop:DT19thm}.
To get the desired upper bound in $\mathsf{LimSup}_p$, it suffices to show that 
\begin{equation}\label{eq:temp6}
\liminf_{\epsilon \to 0}\limsup_{t \to \infty} \frac{1}{t} \log\Big(\int_{-\infty}^{\infty}  e^{-\frac{(1-\epsilon) p x^2}{2t}} M^{f_t}_p(t,x) dx\Big) \leq g(p).
\end{equation}
 For showing \eqref{eq:temp6},  
we use the property \eqref{eq:condition 3} of the KPZ data $\big(g,\{f_t\}_{t\geq 0}\big)$. By H\"{o}lder's inequality, $\mathbb{E}\big[e^{p f_t (x)}\big]$ is bounded above by $(\mathbb{E}[e^{p(1 +\epsilon) f_t (x)}])^{1/(1+\epsilon)}$ which we can bound by $1 + \mathbb{E}[e^{p(1+\epsilon) f_t (x)}]$. Applying this upper bound into the left hand side of \eqref{eq:temp6},  
\begin{align*}
\text{l.h.s. of \eqref{eq:temp6}}
&\leq \liminf_{\epsilon \to 0}\limsup_{t \to \infty} \frac{1}{t} \log\Big(\int_{-\infty}^{\infty}  e^{-\frac{(1-\epsilon) p x^2}{2t}} \big(1 + M^{f_t}_{p(1+\epsilon)}(t,x)\big) dx\Big)\\
&\leq \liminf_{\epsilon \to 0} \limsup_{t \to \infty} \frac{1}{t} \log \Big( \sqrt{\frac{2\pi t}{(1-\epsilon) p}} + \int_{-\infty}^{\infty} e^{-\frac{(1-\epsilon) p x^2}{2t}} M^{f_t}_{p(1+\epsilon)}(t,x) dx\Big)\\
&\leq g(p).
\end{align*}
We have used the property $g(p) \geq 0$ from Lemma \ref{lem:condition} (i) and \eqref{eq:condition 3} in the last line. This completes the proof when $p>1$.  
\smallskip 

\subsubsection{\emph{Stage 2:}}
The derivation of $\mathsf{LimSup}_p$ for $p\in (0,1]$  depends on Proposition~\ref{prop:NWcontrol} and~\ref{prop:InitControl}. We first state these propositions. We will use the propositions to prove $\mathsf{LimSup}_p$ for $p\in (0,1]$ and then prove them. 
\begin{prop}\label{prop:NWcontrol}
Fix any $\nu\in (0,1)$. For any $u,v\in \mathbb{R}$, define a random function $S_{[u,v]}:(0,\infty)\to \mathbb{R}$ as 
\begin{equation}\label{eq:limsup<1 1}
S_{[u,v]} (t) := \sup_{x \in [u, v]} \Big(\mathcal{H}^{\mathbf{nw}} (t, x) - \mathcal{H}^{\mathbf{nw}}(t, u) + \frac{(x-u) u }{t} -  \frac{\nu (x - u)^2}{2}\Big).
\end{equation} 
Let $\{\theta_{n}\}_{n\in \mathbb{Z}}$ be a sequence of grid points such that the sequence  $\{f_t\}_{t>0}$ satisfies \eqref{eq:ModCont}.   
Then, we have the following: 
\begin{enumerate}
\item[(i)] For all $n\in \mathbb{Z}$,
\begin{equation}\label{eq:DistId}
\Big(\mathcal{H}^{\mathbf{nw}}(t, \theta_n) + \frac{\theta_n^2}{2t}, S_{[\theta_n,\theta_{n+1}]} (t)\Big) \overset{d}{=}  \Big(\mathcal{H}^{\mathbf{nw}}(t, 0), S_{[0,\theta_{n+1}-\theta_{n}]} (t)\Big).
\end{equation}
\item[(ii)] For all $p\in \mathbb{R}_{>0}$,
\begin{align}\label{eq:limsup>1 2}
&\limsup_{t \to \infty} \frac{1}{t} \log \mathbb{E}\Big[\mathcal{Z}^{\mathbf{nw}}(t, 0)^p e^{p S_{[0,1]} (t)}\Big] \leq \frac{p^3 - p}{24}.
\end{align}
\end{enumerate}
\end{prop}

\begin{prop}\label{prop:InitControl}
For any $n\in \mathbb{Z}$, we define $\mathbf{E}^{(n)}_{t,p}:=\mathbb{E}[(\int^{\theta_{n+1}}_{\theta_{n}}e^{f_t (x)} dx)^p]$. Then,
\begin{align}\label{eq:limsup<1 2}
&\limsup_{t \to \infty} \frac{1}{t} \log \Bigg(\sum_{n \in \mathbb{Z}_{\geq 0}} e^{-\frac{p\theta^2_{n}}{2t}}\mathbf{E}^{(n)}_{t,p} +\sum_{n \in \mathbb{Z}_{<0}} e^{-\frac{p\theta^2_{n+1}}{2t}} \mathbf{E}^{(n)}_{t,p}\Bigg) \leq g(p).
\end{align}
\end{prop}

\textsc{Proof of $\mathsf{LimSup}_{p}$ for $p\in (0,1]$:} Fix $p\in (0,1]$. We show that there exists $C=C(p)>0$ such that for all $t>0$,  
\begin{align}\label{eq:p<1Split}
\mathbb{E}\Big[\mathcal{Z}^{f_t} (t, 0)^p\Big]
\leq C \mathbb{E}\Big[\mathcal{Z}^{\mathbf{nw}}(t, 0)^p e^{p S_{[0,1]} (t)}\Big] \Bigg(\sum_{n \in \mathbb{Z}_{\geq 0}} e^{-\frac{p\theta^2_n}{2t}} \mathbf{E}^{(n)}_{t,p}+ \sum_{n \in \mathbb{Z}_{<0}} e^{-\frac{p\theta^2_{n+1}}{2t}} \mathbf{E}^{(n)}_{t,p}\Bigg).
\end{align}
From the above inequality, we first show how $\mathsf{LimSup}_p$ follows. By taking logarithms of both sides of the above inequality, dividing them by $t$ and letting $t\to \infty$, we get $\mathsf{LimSup}_p$ once the following inequalities are satisfied
$$\limsup_{t \to \infty} \frac{1}{t} \log \mathbb{E}\Big[\mathcal{Z}^{\mathbf{nw}}(t, 0)^p e^{p S_{[0,1]} (t)}\Big] \leq \frac{p^3 - p}{24}, \quad \limsup_{t \to \infty} \frac{1}{t} \log \Bigg(\sum_{n \in \mathbb{Z}_{\geq 0}} e^{-\frac{p\theta^2_n}{2t}}\mathbf{E}^{(n)}_{t,p} +\sum_{n \in \mathbb{Z}_{<0}} e^{-\frac{p\theta^2_{n+1}}{2t}} \mathbf{E}^{(n)}_{t,p}\Bigg) \leq g(p).$$
But, these two inequalities are given by \eqref{eq:limsup>1 2} and \eqref{eq:limsup<1 2} of Proposition~\ref{prop:NWcontrol} and~\ref{prop:InitControl} respectively. This completes the proof of $\mathsf{LimSup}_p$ when $p\in (0,1]$ modulo \eqref{eq:p<1Split} which we prove as follows.

By the convolutional formula of Proposition~\ref{prop:convolution}, it suffices to show \eqref{eq:p<1Split} with $\mathbb{E}[(\int \mathcal{Z}^{\mathbf{nw}}(t, x) e^{f_t (x)} dx)^p]$ in place of $\mathbb{E}[\mathcal{Z}^{f_t} (t, 0)^p]$. Owing to the subadditivity of function $g(x) = x^{p}$ for $x>0$ and $p\in (0,1]$, 
\begin{align}\label{eq:SubAdd}
 \mathbb{E}\Big[\big(\int \mathcal{Z}^{\mathbf{nw}}(t, x) e^{f_t (x)} dx\big)^p\Big]= \mathbb{E}\Big[\Big(\sum_{n \in \mathbb{Z}} \int_{\theta_{n}}^{\theta_{n+1}} \mathcal{Z}^{\mathbf{nw}}(t, x) e^{f_t (x)} dx\Big)^p\Big]\leq \sum_{n \in \mathbb{Z}} \mathbb{E}\Big[\Big(\int_{\theta_n}^{\theta_{n+1}} \mathcal{Z}^{\mathbf{nw}}(t, x) e^{f_t (x)} dx  \Big)^p\Big].
\end{align} 
Note that \eqref{eq:p<1Split} follows from the above inequality if there exists $C=C(p)>0$ such that 
\begin{align}
\mathbb{E}\Big[\Big(\int_{\theta_n}^{\theta_{n+1}} \mathcal{Z}^{\mathbf{nw}}(t, x) e^{f_t (x)} dx \Big)^p\Big]\leq C   \mathbb{E}\Big[\mathcal{Z}^{\mathbf{nw}}(t, 0)^p e^{p S_{[0,1]} (t)}\Big] \mathbf{E}^{(n)}_{t,p}\times \begin{cases}
e^{-\frac{p \theta^2_n}{2t}} & n\geq 0,\\
e^{-\frac{p\theta^2_{n+1}}{2t}} & n<0.
\end{cases}\label{eq:SplitStage2}
\end{align}
holds for all $n\in \mathbb{Z}$. We show this bound below.

We first show \eqref{eq:SplitStage2} for $n\geq 0$. Recall the definition of $S_n(t)$ from \eqref{eq:limsup<1 1}. Since $S_{[\theta_n,\theta_{n+1}]}(t)$ is greater than $\mathcal{H}^{\mathbf{nw}} (t, x) - \mathcal{H}^{\mathbf{nw}} (t, \theta_{n}) + (x-\theta_n) \theta_n/t - \nu(x-\theta_n)^2/2$ for any $x\in [\theta_n,\theta_{n+1}]$,  we may write 
\begin{align}
\mathcal{H}^{\mathbf{nw}}(t,x) \leq \mathcal{H}^{\mathbf{nw}}(t,\theta_n) + S_{[\theta_n,\theta_{n+1}]}(t) -\frac{(x-\theta_n)\theta_n}{t}+ \frac{\nu(x-\theta_n)^2}{2}.
\end{align}
Exponentiating both sides of the inequality yields 
\begin{equation}\label{eq:Expo}
\mathcal{Z}^{\mathbf{nw}}(t, x) 
\leq \mathcal{Z}^{\mathbf{nw}}(t, \theta_n) e^{S_{[\theta_n,\theta_{n+1}]} (t)} e^{\frac{\nu(x-\theta_n)^2}{2}} e^{-\frac{(x-\theta_n)\theta_{n}}{t}}\leq C\mathcal{Z}^{\mathbf{nw}}(t, \theta_n) e^{S_{[\theta_n,\theta_{n+1}]} (t)},
\end{equation}
where the last inequality follows since $\exp(2^{-1}\nu (x-\theta_n)^2-t^{-1}(x-\theta_{n})\theta_{n})$is upper bounded by a constant over $x \in [\theta_n, \theta_{n+1}]$ (recall that we set $|\theta_{n+1} - \theta_n| \leq 1$). Bounding $\mathcal{Z}^{\mathbf{nw}}(t, x)$ with $C\mathcal{Z}^{\mathbf{nw}}(t, \theta_n) e^{S_{[\theta_n,\theta_{n+1}]} (t)}$ yields
\begin{align*}
\Big(\int_{\theta_n}^{\theta_{n+1}} \mathcal{Z}^{\mathbf{nw}}(t, x) e^{f_t (x)} dx\Big)^p 
& \leq C \mathcal{Z}^{\mathbf{nw}}(t, \theta_{n})^p e^{p S_{[\theta_n,\theta_{n+1}]} (t)} \Big(\int_{\theta_n}^{\theta_{n+1}} e^{f_t (x)} dx\Big)^p.
\end{align*}
Taking the expectation for both sides in the above display and using the independence between $\mathcal{Z}^{\mathbf{nw}}(t,\cdot)$ and $f_t(\cdot)$ shows
\begin{align}
\mathbb{E}\Big[\Big(\int_{\theta_n}^{\theta_{n+1}} \mathcal{Z}^{\mathbf{nw}}(t, x) e^{f_t (x)} dx\Big)^p\Big] 
&\leq C\mathbb{E}\Big[\mathcal{Z}^{\mathbf{nw}}(t, \theta_n)^p e^{p S_{[\theta_n,\theta_{n+1}]} (t)}\Big]\mathbf{E}^{(n)}_{t,p}\nonumber\\&= C\mathbb{E}\Big[\big(\mathcal{Z}^{\mathbf{nw}}(t, \theta_n)e^{\frac{\theta^2_n}{2t}}\big)^p e^{p S_{[\theta_n,\theta_{n+1}]} (t)}\Big]\mathbf{E}^{(n)}_{t,p}e^{-\frac{p\theta^2_n}{2t}}\label{eq:npos}.
\end{align}
By \eqref{eq:DistId} of Proposition~\ref{prop:NWcontrol}, $(\mathcal{Z}^{\mathbf{nw}}(t, \theta_n)e^{\frac{\theta^2_n}{2t}}, S_{[\theta_n,\theta_{n+1}]} (t))$ is same in distribution with $(\mathcal{Z}^{\mathbf{nw}}(t, 0), S_{[0,\theta_{n+1}-\theta_n]} (t))$. Note that $e^{pS_{[0,\theta_{n+1}-\theta_n]}(t)}$ is bounded above by $e^{pS_{[0,1]}(t)}$ since $|\theta_{n+1}-\theta_{n}| \leq 1$. Thus, the right hand side of the above display is less than $C\mathbb{E}[\mathcal{Z}^{\mathbf{nw}}(t,0)^p e^{pS_{[0,1]}(t)}]\mathbf{E}^{(n)}_{t,p}e^{-\frac{p\theta^2_n}{2t}}$. This shows \eqref{eq:SplitStage2} for $n\geq 0$.  

We turn to prove \eqref{eq:SplitStage2} for $n<0$. The key part of the proof relies on the fact that the law of $\mathcal{Z}^{\mathbf{nw}}(t,\cdot)$ is invariant under the reflection w.r.t. $0$, i.e., $\{\mathcal{Z}^{\mathbf{nw}}(t,x):x\geq 0\}$ is same in distribution with $\{\mathcal{Z}^{\mathbf{nw}}(t,x):x\leq 0\}$.  By this reflection invariance of the law of $\mathcal{Z}^{\mathbf{nw}}(t,\cdot)$, it suffices to bound $\mathbb{E}[(\int_{-\theta_{n+1}}^{-\theta_n} \mathcal{Z}^{\mathbf{nw}}(t, x) e^{f_t (-x)} dx)^p]$ instead of $\mathbb{E}[(\int_{\theta_n}^{\theta_{n+1}} \mathcal{Z}^{\mathbf{nw}}(t, x) e^{f_t (x)} dx)^p]$. Note that $-\theta_{n+1}\geq 0$ for any $n\in \mathbb{Z}_{<0}$. By \eqref{eq:Expo}, we can bound $\mathcal{Z}^{\mathbf{nw}}(t, x)$ by $C\mathcal{Z}^{\mathbf{nw}}(t,-\theta_{n+1})e^{S_{[-\theta_{n+1}, -\theta_n]}(t)}$ for some constant $C=C(p,\nu)>0$ for any $x\in [-\theta_{n+1},-\theta_n]$. This allows us to write 
\begin{align}
\mathbb{E}\Big[ &\Big(\int_{-\theta_{n+1}}^{-\theta_n} \mathcal{Z}^{\mathbf{nw}}(t, x) e^{f_t (-x)} dx\Big)^p\Big]\nonumber\\&\leq C \mathbb{E}\Big[\big(\mathcal{Z}^{\mathbf{nw}}(t, -\theta_{n+1})e^{\frac{\theta^2_{n+1}}{2t}}\big)^p e^{p S_{[-\theta_{n+1}, -\theta_n]} (t)}\Big]\mathbb{E}\big[\big(\int^{-\theta_n}_{-\theta_{n+1}}e^{f_t(-x)}dx\big)^p\big]e^{-\frac{p\theta^2_{n+1}}{2t}}\label{eq:Stage2last}
\end{align}
in the same way as in \eqref{eq:npos}. In what follows, we explain how to obtain \eqref{eq:SplitStage2} for $n<0$ from the above inequality. We first bound $\mathbb{E}[(\mathcal{Z}^{\mathbf{nw}}(t, -\theta_{n+1})e^{\frac{\theta^2_{n+1}}{2t}})^p e^{p S_{[-\theta_{n+1}, -\theta_n]} (t)}]$ by $\mathbb{E}[\mathcal{Z}^{\mathbf{nw}}(t,0)^p e^{pS_{[0,1]}(t)}]$ in the right side of \eqref{eq:Stage2last} and this substitution is justified by \eqref{eq:DistId} of Proposition~\ref{prop:NWcontrol}. Next, we identify $\mathbb{E}[(\int^{-\theta_n}_{-\theta_{n+1}}e^{f_t(-x)}dx)^p]$ with $\mathbf{E}^{(n)}_{t,p}$ in \eqref{eq:Stage2last} by change of variable inside the integral. Combining the outcomes of these two steps with the fact that left side of \eqref{eq:Stage2last} is equal to $\mathbb{E}[(\int_{\theta_n}^{\theta_{n+1}} \mathcal{Z}^{\mathbf{nw}}(t, x) e^{f_t (x)} dx)^p]$ shows \eqref{eq:SplitStage2} for $n<0$. This completes the proof of the desired result. 


\begin{proof}[Proof of Proposition~\ref{prop:NWcontrol}] (i)
Recall the definition of $S_{[\theta_n,\theta_{n+1}]}(t)$  from \eqref{eq:limsup<1 1}. Rewriting $\frac{(x - \theta_n) \theta_n }{t}$ into $\frac{x^2}{2t} - \frac{\theta^2_n}{2t} - \frac{(x-\theta_n)^2}{2t}$, we get
\begin{align}
S_{[\theta_n,\theta_{n+1}]} (t) 
&= \sup_{x \in [\theta_n,\theta_{n+1}]} \Big(\mathcal{H}^{\mathbf{nw}} (t, x) + \frac{x^2}{2t} - \mathcal{H}^{\mathbf{nw}}(t, \theta_n) - \frac{\theta^2_n}{2t} - \frac{(x-\theta_n)^2}{2t} -  \frac{\nu (x - \theta_n)^2}{2}\Big)\nonumber\\
&= \sup_{x \in [0, \theta_{n+1} - \theta_n]} \Big(\mathcal{H}^{\mathbf{nw}} (t, x+\theta_n) + \frac{(x+\theta_n)^2}{2t}  - \mathcal{H}^{\mathbf{nw}}(t, \theta_n) - \frac{\theta^2_n}{2t}  - \frac{x^2}{2t} -  \frac{\nu x^2}{2}\Big)\label{eq:SnRewrite} ,
\end{align}
where second line is due to a change of variable $x \to x+\theta_n$.
 By Proposition \ref{prop:stationarity}, for any fixed $t>0$, the process $\mathcal{H}^{\mathbf{nw}}(t, x) + \frac{x^2}{2t}$ is stationary in $x$. This implies $\{\mathcal{H}^{\mathbf{nw}} (t, x+\theta_n)+\frac{(x+\theta_n)^2}{2t} :x\in [0,\theta_{n+1}-\theta_n]\}$ is same in distribution with $\{\mathcal{H}^{\mathbf{nw}} (t, x) +\frac{x^2}{2t}\in [0,\theta_n]\}$ for any $n\in \mathbb{Z}.$ Note that 
\begin{align}
S_{[0,\theta_{n+1}-\theta_n]} (t) & = \sup_{x \in [0, \theta_{n+1}-\theta_n]} \Big(\mathcal{H}^{\mathbf{nw}} (t, x) - \mathcal{H}^{\mathbf{nw}}(t, 0) - \frac{\nu x^2}{2}\Big)\nonumber\\& = \sup_{x \in [0, \theta_{n+1}-\theta_n]} \Big(\mathcal{H}^{\mathbf{nw}} (t, x) + \frac{x^2}{2t} - \mathcal{H}^{\mathbf{nw}}(t, 0) - \frac{x^2}{2t} - \frac{\nu x^2}{2}\Big)\label{eq:S0} .
\end{align} 
Now, \eqref{eq:DistId} follows by comparing \eqref{eq:S0} with \eqref{eq:SnRewrite} and using the stationarity of $\mathcal{H}^{\textbf{nw}} (t, x) + \frac{x^2}{2t}$, which implies the equivalence of the law of $\{\mathcal{H}^{\mathbf{nw}} (t, x+\theta_n) +\frac{(x+\theta_n)^2}{2t}:x\in [0,\theta_{n+1}-\theta_{n}]\}$ with $\{\mathcal{H}^{\mathbf{nw}} (t, x) +\frac{x^2}{2t}: x\in [0,\theta_{n+1}-\theta_{n}]\}$. 
\smallskip 
\\

\noindent (ii)
 For any $\epsilon>0$, we seek to show that there exists $C=C(p,\nu,\epsilon)>0$ such that 
\begin{align}
\mathbb{E}\Big[\mathcal{Z}^{\mathbf{nw}}(t, 0)^p e^{p S_{[0,1]} (t)}\Big] \leq C \Big(\mathbb{E}\Big[\mathcal{Z}^{\mathbf{nw}}(t, 0)^{p+\epsilon} \Big]\Big)^{\frac{p}{p+\epsilon}}.\label{eq:Holder}
\end{align} 
Before proceeding to its proof, we first explain how the above inequality implies \eqref{eq:limsup<1 2}. Taking the logarithm and then dividing both side of above display by $t$ and letting $t \to \infty$, we get 
\begin{align*}
\limsup_{t \to \infty} \frac{1}{t} \log \mathbb{E}\Big[\mathcal{Z}^{\mathbf{nw}}(t, 0)^p e^{p S_{[0,1]} (t)}\Big]
&\leq \frac{p}{p+\epsilon} \limsup_{t \to \infty} \frac{1}{t}\log \mathbb{E}\Big[\mathcal{Z}^{\mathbf{nw}}(t, 0)^{p+\epsilon}\Big] = \frac{p (p+\epsilon)^2 - p}{24},
\end{align*}
where the last equality follows from Proposition \ref{prop:DT19thm}. Letting  $\epsilon \to 0$ in the last display, we get the desired \eqref{eq:limsup>1 2}.

It remains to show \eqref{eq:Holder} which is proved as follows. By H\"{o}lder's inequality, for arbitrary $\epsilon > 0$,
\begin{align}\label{eq:MyIntermediate}
\mathbb{E}\Big[\mathcal{Z}^{\mathbf{nw}}(t, 0)^p e^{p S_{[0,1]} (t)}\Big] &\leq \Big(\mathbb{E}\Big[\mathcal{Z}^{\mathbf{nw}}(t, 0)^{p+\epsilon} \Big]\Big)^{\frac{p}{p+\epsilon}} \Big(\mathbb{E}\Big[e^{\frac{p(p+\epsilon)}{\epsilon} S_{[0,1]} (t)} \Big]\Big)^{\frac{\epsilon}{p+\epsilon}} .
\end{align}
From the last inequality, \eqref{eq:Holder} follows if we can bound $\mathbb{E}[e^{\frac{p(p+\epsilon)}{\epsilon} S_{[0,1]} (t)}]$  by some constant $C=C(p,\nu, \epsilon)>0$. We will now accomplish this using the tail probability bound of $S_{[0,1]}(t)$. By \eqref{eq:increment upper tail} of Proposition~\ref{prop: CGH 4.4}, we know that for any fixed $\delta>0$, there exist $s_0=s_0(\delta,\nu)>0$ and $c=c(\delta,\nu)>0$ such that $\mathbb{P}(S_{[0,1]}(t)\geq s)\leq \exp(-cs^{9/8-\delta})$ for all $s\geq s_0$ and $t>1$. We choose $\delta=\frac{1}{17}$. One may notice that $\frac{9}{8}-\frac{1}{17}>1+\frac{1}{17}$. With this computation and tail bound of $S_{[0,1]}(t)$ in hand, we write 
\begin{align}
\mathbb{E}[e^{\frac{p(p+\epsilon)}{\epsilon} S_{[0,1]} (t)}]\leq e^{\frac{p(p+\epsilon)}{\epsilon}s_0}+ \int^{\infty}_{s_0} e^{\frac{p(p+\epsilon)}{\epsilon}s -cs^{1+\frac{1}{17}}}ds. 
\end{align}
The right hand side of the above inequality is a finite constant whose value would depend on $p,\nu,\epsilon$. Combining this with \eqref{eq:MyIntermediate} yields the proof of \eqref{eq:Holder}.


\end{proof}

\bigskip

\begin{proof}[Proof of Proposition~\ref{prop:InitControl}]
Recall the notation $M^{f_t}_p(t,x)$ from Definition~\ref{Def:HypDef}. We will prove \eqref{eq:limsup<1 2} using the following claim: Fix arbitrary $\e > 0$, there exist $C_1=C_1(p,\epsilon)$ and $C_2= C_2(p,\epsilon)>0$ such that for all $t>1$,
\begin{align}
\sum_{n \in \mathbb{Z}_{\geq 0}} e^{-\frac{p\theta^2_n}{2t}} \mathbf{E}^{(n)}_{t,p} + \sum_{n \in \mathbb{Z}_{<0}} e^{-\frac{p\theta^2_{n+1}}{2t}} \mathbf{E}^{(n)}_{t,p}\leq C_1t^{\frac{1}{2(1-\beta)}} + C_2   \int_{\mathbb{R}} e^{-\frac{p(1-\epsilon) x^2}{2t}} M^{f_t}_{p(1+\epsilon)}(t,x) dx\label{eq:Stage2Cruc} ,
\end{align}
where $\beta\in(0,1)$ is the same constant as in the \emph{pseudo-stationarity} condition of Definition~\ref{Def:HypDef}. Recall $\mathbf{E}^{(n)}_{t,p}=\mathbb{E}[(\int^{\theta_{n+1}}_{\theta_{n}}e^{f_t (x)} dx)^p]$.
After proving \eqref{eq:limsup<1 2} which we do as follows, we will proceed to prove the above inequality. Taking the logarithm of both sides of \eqref{eq:Stage2Cruc} and noting that $\log(c_1a+c_2b)\leq \log(\max\{c_1,c_2\})+\log 2a+\max\{\log a, \log b\}$ for any $a\geq 1, b>0$, $c_1,c_2>0$, we get 
\begin{align}
\log\big(\text{r.h.s. of \eqref{eq:Stage2Cruc}})\leq \log(\max\{C_1,C_2\})+
\log 2t^{\frac{1}{2(1-\beta)}} +\max\Big\{\log t^{\frac{1}{2(1-\beta)}}, \log\Big(\int_{\mathbb{R}} e^{-\frac{p(1-\epsilon) x^2}{2t}} M^{f_t}_{p(1+\epsilon)}(t,x)\Big)\Big\}.\label{eq:Stage2Cruc2}
\end{align}
Now, we divide both sides by $t$ and let $t \to \infty$. On doing so, we claim that the limit of the right hand side is less than $g(p)$. To see this, we first write 
\begin{align*}
\limsup_{t\to \infty}&\frac{1}{t}\max\Big\{\log t^{\frac{1}{2(1-\beta)}}, \log(\int_{\mathbb{R}} \exp(-p(1-\epsilon) x^2/2t) M^{f_t}_{p(1+\epsilon)}(t,x))dx)\Big\}
\\& \leq \max\Big\{\limsup_{t\to \infty}\frac{1}{t}\log t^{\frac{1}{2(1-\beta)}}, \limsup_{t\to \infty}\frac{1}{t}\log(\int_{\mathbb{R}} \exp(-p(1-\epsilon) x^2/2t) M^{f_t}_{p(1+\epsilon)}(t,x))dx)\Big\} .
\end{align*} 
Then, we note 
$$\lim_{t\to \infty}\frac{1}{t}\log(\max\{C_1, C_2\}) =0, \quad \lim_{t\to \infty}\frac{1}{t}\log t^{\frac{1}{2(1-\beta)}}= 0, \quad \liminf_{\epsilon\to 0}\limsup_{t\to \infty}\frac{1}{t}\log\Big(\int_{\mathbb{R}} e^{-\frac{p(1-\epsilon) x^2}{2t}} M^{f_t}_{p(1+\epsilon)}(t,x) dx\Big) \leq g(p) ,$$
where the last inequality follows by applying \eqref{eq:condition 3}.  Substituting these limiting results into the right hand side of \eqref{eq:Stage2Cruc2} and using the non-negativitiy of $g(p)$ (Lemma \ref{lem:condition} (i)) conclude \eqref{eq:limsup<1 2}. This proves Proposition~\ref{prop:InitControl} modulo \eqref{eq:Stage2Cruc} which is proved as follows.   
 

Note that \eqref{eq:Stage2Cruc} bounds a discrete sum by an integral. This passage from discrete to continuum requires a locally uniform control on the discrete summands of \eqref{eq:Stage2Cruc2} which we seek to extract from the tail bounds of \eqref{eq:ModCont}. To this aim, for any $n\in \mathbb{Z}$,
\begin{equation}\label{eq:TVdef}
\mathrm{TV}_{f_t}(n):= \sup_{y \in [\theta_n, \theta_{n+1}]} |f_t (x) - f_t (\theta_n)|.
\end{equation}
Since $\mathrm{TV}_{f_t}(n)$ is the supremum of $|f_t (x) - f_t (\theta_n)|$ as $x$ varies in $ [\theta_n, \theta_{n+1}]$, we may bound $f_t (x)$ by $ \mathrm{TV}_{f_t}(n) + f_t (\theta_n)$ for all $x\in [\theta_n,\theta_{n+1}]$. This allows us to bound $\mathbb{E}[(\int_{\theta_n}^{\theta_{n+1}} e^{f_t (x)} dx)^p]$ by $\mathbb{E}[e^{p f_t (\theta_n)}  e^{p\mathrm{TV}_{f_t} (n)}]$. Hereafter, we prove \eqref{eq:Stage2Cruc} in two \emph{steps}. \emph{Step 1} will show that there exist $c_1=c_1(p,\epsilon)>0$ and $c_2=c_2(p,\epsilon)>0$ such that the following inequality 
\begin{align}
\mathbb{E}\big[e^{p f_t (\theta_n)}  e^{p\mathrm{TV}_{f_t} (n)}\big]\leq c_1\big(1+ \mathbb{E}\big[e^{p(1+\epsilon/2)f_t(\theta_n)}\big]\big)\leq c_2\big(1+ \int^{\theta_{n+1}}_{\theta_n}M^{f_t}_{p(1+\epsilon)}(t,x)dx\big)\label{eq:Mptx}
\end{align}
holds for all $n\in \mathbb{Z}$. In \emph{Step 2}, we will prove the following: there exist $c^{\prime}_1=c^{\prime}_1(p,\epsilon)>0$ and $c^{\prime}_2=c^{\prime}_2(p,\epsilon)>0$ such that for all $t>1$
\begin{align}
\sum_{n\in \mathbb{Z}_{\geq 0}} e^{-\frac{p\theta^2_n}{2t}} &\Big(1+ \int^{\theta_{n+1}}_{\theta_n}M^{f_t}_{p(1+\epsilon)}(t,x)dx\Big)+ \sum_{n\in \mathbb{Z}_{< 0}} e ^{-\frac{p\theta^2_{n+1}}{2t}} \Big(1+ \int^{\theta_{n+1}}_{\theta_n}M^{f_t}_{p(1+\epsilon)}(t,x)dx\Big)\nonumber\\&\leq c^{\prime}_1t^{\frac{1}{2(1-\beta)}}+c^{\prime}_2\int_{\mathbb{R}}e^{-\frac{p(1-\epsilon)x^2}{2t}} M^{f_t}_{p(1+\epsilon)}(t,x) dx \label{eq:DToC}
\end{align}
where $\beta\in (0,1]$ is the same constant as in the \emph{pseudo-stationarity} condition of Definition~\ref{Def:HypDef}.
Combining \eqref{eq:Mptx} with \eqref{eq:DToC} yields \eqref{eq:Stage2Cruc}.
\smallskip 
\\


\noindent \emph{Step 1:} We start with showing the first inequality of \eqref{eq:Mptx}. By denoting $X:= \exp(p f_t (\theta_n))$ and $W:= \exp(p\mathrm{TV}_{f_t} (n))$, we apply H\"older's inequality to bound $\mathbb{E}[XW]$ by $(\mathbb{E}[X^{(1+\epsilon/2)}])^{1/(1+\epsilon/2)}(\mathbb{E}[W^{(2+\epsilon)/\epsilon}])^{\epsilon/(2+\epsilon)}$. The first inequality of \eqref{eq:Mptx} will follow from this upper bound once we show 
\begin{align}\label{eq:TwoIneq}
(\mathbb{E}[X^{(1+\epsilon/2)}])^{1/(1+\epsilon/2)}\leq 1+ \mathbb{E}[X^{(1+\epsilon/2)}], \qquad \qquad \qquad (\mathbb{E}[W^{(2+\epsilon)/\epsilon}])^{\epsilon/(2+\epsilon)}\leq c_1
\end{align}
for some constant $c_1=c_1(p,\epsilon)>0$. The left hand side inequality is straightforward since $x^{a}\leq \max\{1,x\}$ for any $x>0$ and $a\in (0,1)$. To prove the right hand side inequality, we use the tail bound $\mathrm{TV}_{f_t} (n)$. By \eqref{eq:ModCont}, we know that for any $\delta>0$, there exist $s_0=s_0(\delta)>0$ and $c=c(\delta)>0$ such that $\mathbb{P}(\mathrm{TV}_{f_t} (n)>s)\leq \exp(-cs^{1+\delta})$ for all $s\geq s_0$ and $t>0$. With this tail estimate, we can bound $\mathbb{E}[W^{(2+\epsilon)/\epsilon}]$ by $\exp(ps_0(2+\epsilon)/\epsilon)+ \int^{\infty}_{s_0}\exp(ps(2+\epsilon)/\epsilon- cs^{1+\delta})ds$ from above. Since this upper bound is a constant which only depends on $p,\delta,\epsilon$, we get the right hand side inequality of the above display. Combining both proofs shows the first inequality of \eqref{eq:Mptx}.

Now, we show the second inequality of \eqref{eq:Mptx}. Since $f_t(\theta_n)$ is bounded above by $f_t(x)+\mathrm{TV}_{f_t}(n)$ for all $n\in \mathbb{Z}$ and $x\in [\theta_n,\theta_{n+1}]$, we get  
\begin{equation*}
\mathbb{E}\Big[e^{p(1 + \epsilon/2) f_t (\theta_n)}\Big] \leq \int_{\theta_n}^{\theta_{n+1}} \mathbb{E}\Big[e^{p (1 +\epsilon/2) f_t (x)} e^{p(1 + \epsilon/2) \mathrm{TV}_{f_t} (n)}\Big] dx .
\end{equation*}
From this upper bound, the second inequality of \eqref{eq:Mptx} follows if we can show that $\mathbb{E}\big[\exp( (1 +\epsilon/2) f_t (x)) \exp\big(p(1 + \epsilon/2) \mathrm{TV}_{f_t} (n)\big)\big]$ is bounded by $c_1+ c_2M^{f_t}_{p(1+\epsilon)}(t,x)$ for some constants $c_1=c_1(p,\epsilon)>0$ and $c_2= c_2(p,\epsilon)>0$. The proof of this upper bound is similar in spirit to the argument in the previous paragraph. We claim and prove this bound as follows. By denoting $X^{\prime}:= \exp(pf_t(x)(1+\epsilon/2))$ and $W^{\prime}:= \exp\big(p(1 + \epsilon/2) \mathrm{TV}_{f_t} (n)\big)$, we use H\"older's inequality to bound $\mathbb{E}[X^{\prime}W^{\prime}]\leq (E[(X^{\prime})^{u}])^{1/u}(\mathbb{E}[(W^{\prime})^v]^{1/v})$ where $u= (1+\epsilon)/(1+\epsilon/2)$ and $u^{-1}+v^{-1}=1$. Using similar argument as in the proof of \eqref{eq:TwoIneq}, we bound $(E[(X^{\prime})^{u}])^{1/u}$ by $1+E[(X^{\prime})^{u}]$ and $(\mathbb{E}[(W^{\prime})^v]^{1/v})$ by some constant which only depends on $p,\epsilon$. Combining these shows that $\mathbb{E}[X^{\prime}W^{\prime}]$ is bounded above by $c(1+E[X^{\prime})^{u}])$. This proves our claim since $E[(X^{\prime})^{u}]=\mathbb{E}[\exp(p(1+\epsilon)f_t(x))]= M^{f_t}_{p(1+\epsilon)}(t,x)$. As a consequence, we get the second inequality of \eqref{eq:Mptx}.
\smallskip 
\\

\noindent \emph{Step 2:} To prove \eqref{eq:DToC}, we first claim $\sum_{n\in \mathbb{Z}_{\geq 0}} \exp(-\frac{p\theta^2_n}{2t})$ and $\sum_{n\in \mathbb{Z}_{<0}} \exp(-\frac{p\theta^2_{n+1}}{2t})$ can be bounded by $Ct^{1/(2(1-\beta))}$ for all $t>1$ where $\beta$ is the same constant as in the \emph{pseudo-stationarity condition} of Definition~\ref{Def:HypDef} and the constant $C>0$ depends on $p$ and $\beta$. Note that $1\geq |\theta_{n+1}-\theta_n|\geq \min\{1,c|n|^{-\beta}\}$ for some $c>0$, $\beta\in (0,1)$. Therefore, there exists $c_1,c_2>0$ such that $c_1n\geq |\theta_n|\geq c_2 |n|^{1-\beta}$ for all $n\in \mathbb{Z}$. Due to the last inequality, we may write $$|\theta_{n+1}-\theta_{n}|\geq D\max\{1,|\theta_n|^{-\frac{\beta}{(1-\beta)}}\}\geq D \max\{1,(|x|+1)^{-\frac{\beta}{(1-\beta)}}\},\quad \forall x\in [\theta_n,\theta_{n+1}]$$ for some constant $D>0$. 
Since $\exp(-\frac{p\theta^2_n}{2t})$ and $\exp(-\frac{p\theta^2_{m+1}}{2t})$ decreases as $n\uparrow \infty$ and $m\downarrow -\infty$ bounding the Riemann sum with its integral approximation yields 
\begin{align*}
\max\Big\{\sum_{n\in \mathbb{Z}_{\geq 0}} \exp(-\frac{p\theta^2_n}{2t}), \sum_{n\in \mathbb{Z}_{<0}} \exp(-\frac{p\theta^2_{n+1}}{2t})\Big\}&\leq 1+D^{-1}\int_{\mathbb{R}}(|x|+1)^{\frac{\beta}{1-\beta}}\exp(-\frac{px^2}{2t})dx\\&\leq 1+2^{\frac{\beta}{1-\beta}}D^{-1}\int_{\mathbb{R}}(|x|^{\frac{\beta}{1-\beta}}+1)\exp(-\frac{px^2}{2t})dx ,
\end{align*} 
where the last inequality follows since $(|x|+1)^{\beta/(1-\beta)}$ is bounded by $2^{\beta/(1-\beta)}(|x|^{\beta/(1-\beta)}+1)$. The integral on the right hand side of the above display is bounded by $Ct^{1/(2(1-\beta))}$ when $t>1$ for some constant $C=C(p,\beta)>0$.
This proves our claim. To complete the proof of \eqref{eq:DToC}, it remains to show the following: there exists $t_0=t_0(\epsilon)>0$ such that for all $t>t_0$,
\begin{align}
\sum_{n\in \mathbb{Z}_{\geq 0}} e^{-\frac{p\theta^2_n}{2t}}\int^{\theta_{n+1}}_{\theta_n}M^{f_t}_{p(1+\epsilon)}(t,x)dx &\leq C_1+C_2 \int_{\mathbb{R}}e^{-\frac{p(1-\epsilon)x^2}{2t}}M^{f_t}_{p(1+\epsilon)}(t,x)dx,\label{eq:1} \\\sum_{n\in \mathbb{Z}_{<0}} e^{-\frac{p\theta^2_{n+1}}{2t}}\int^{\theta_{n+1}}_{\theta_n}M^{f_t}_{p(1+\epsilon)}(t,x)dx &\leq C_1+C_2 \int_{\mathbb{R}}e^{-\frac{p(1-\epsilon)x^2}{2t}}M^{f_t}_{p(1+\epsilon)}(t,x)dx
\end{align} 
for some constants $C_1=C_1(p,\epsilon)>0$ and $C_2=C_2(p,\epsilon)$. We only prove \eqref{eq:1}. The proof of the other inequality is similar and details are skipped.

For any given $\epsilon>0$, there exists $n_0=n_0(\epsilon)\in \mathbb{Z}_{\geq 0}$ such that $\theta^2_n\geq (1-\epsilon)x^2$ for all $x\in [\theta_n,\theta_{n+1}]$ and $n\geq n_0$. We write left side of \eqref{eq:1} as  
\begin{align}
\text{l.h.s. of \eqref{eq:1}} = \sum_{0\leq n<n_0(\epsilon)} e^{-\frac{p\theta^2_n}{2t}}\int^{\theta_{n+1}}_{\theta_n}M^{f_t}_{p(1+\epsilon)}(t,x)dx  + \sum_{n\geq n_0(\epsilon)} e^{-\frac{p\theta^2_n}{2t}}\int^{\theta_{n+1}}_{\theta_n}M^{f_t}_{p(1+\epsilon)}(t,x)dx. \label{eq:SplitSum}
\end{align}
We can bound the last term on the right side of the above display by $\int_{\mathbb{R}}\exp(-\frac{p(1-\epsilon)x^2}{2t})M^{f_t}_{p(1+\epsilon)}(t,x)dx$ since $\theta^2_n\geq (1-\epsilon)x^2$ for all $x\in [\theta_n,\theta_{n+1}]$ and $n\geq n_0$. Using the pointwise upper bound on $M^{f_t}_{p(1+\epsilon)}(t,x)$ from \eqref{eq:growth condition}, we can write 
\begin{align*}
\sum_{0\leq n<n_0(\epsilon)} e^{-\frac{p\theta^2_n}{2t}}\int^{\theta_{n+1}}_{\theta_n}M^{f_t}_{p(1+\epsilon)}(t,x)dx\leq \sum_{0\leq n<n_0(\epsilon)} e^{-\frac{p\theta^2_n}{2t}}e^{Cp(1+\epsilon)(1+\theta^{\delta}_{n_0})+ \frac{\alpha \theta^2_{n_0}}{2t}}\leq n_0 e^{Cp(1+\epsilon)(1+\theta^{\delta}_{n_0})+ \alpha p \theta^2_{n_0}} ,
\end{align*}
where the last inequality follows by bounding $e^{-\frac{p\theta^2_{n_0}}{2t}}$ by $1$ for all $0\leq n<n_0$ and taking $t>1$.
Due to the above bound, the first term in the right side of \eqref{eq:SplitSum} is bounded by some constant $C=C(p,\epsilon)>0$. Combining the upper bounds on both summands of \eqref{eq:SplitSum} yields \eqref{eq:1}. This completes the proof of \eqref{eq:DToC} and Proposition~\ref{prop:InitControl}.

\end{proof}

\subsection{Proof of $\mathsf{LimInf}_p$ for all $p\in \mathbb{R}_{>0}$}\label{sec:Lower bound}
Fix any $p,\nu>0$ . Recall the notation $x_{p,\gamma}(t)$ of Lemma~\ref{lem:condition}. For any $0<\epsilon<p/2$ and $\gamma>0$, let $n_{p,\epsilon,\gamma}(t)\in \mathbb{Z}$ be such that $x_{p-\epsilon,\gamma}(t) \in [\theta_{n_{p,\epsilon,\gamma}(t)},\theta_{n_{p,\epsilon,\gamma}(t)+1}]$ where $\{\theta_n\}_{n\in \mathbb{Z}}$ is a sequence of grid points (see Definition~\ref{Def:HypDef}) such that $f_t$ satisfies \eqref{eq:ModCont} for all $t>0$. For notational convenience, we will denote $n_{p,\epsilon,\gamma}(t)$ by $n(t)$ and the interval $[\theta_{n(t)},\theta_{n(t)+1}]$ by $I(t)$. 
For convenience, we use the following shorthand notations:
 \begin{align}
\mathcal{Z}^{\mathbf{nw}}_{p,\epsilon}(t) := \mathcal{Z}^{\mathbf{nw}} (t, x_{p-\epsilon,\gamma} (t)) &e^{\frac{x^2_{p-\epsilon,\gamma}(t)}{2t}},\label{eq:Z(t)defin} \\
Y_{p,\epsilon}(t): = \inf_{x \in I(t)} \Big\{\mathcal{H}^{\mathbf{nw}} (t, x) - \mathcal{H}^{\mathbf{nw}} \big(t, x_{p-\epsilon,\gamma} (t)\big) &+ \frac{\big(x-x_{p-\epsilon,\gamma} (t)\big) x_{p-\epsilon,\gamma} (t)}{t} + \frac{\nu \big(x-x_{p-\epsilon,\gamma} (t)\big)^2}{2}\Big\}.\label{eq:Y(t)defin} 
\end{align}
As in Section~\ref{sec:Upper bound}, we rely on the convolution formula of Proposition~\ref{prop:convolution} to express the moments of $\mathcal{Z}^{f_t}(t,0)$ in terms the moment of a integral involving $\mathcal{Z}^{\mathbf{nw}}(t,\cdot)$ and $e^{f_t(\cdot)}$. To prove $\mathsf{LimInf}_p$, we analyze the expected value of $p$-th moment of this integral over the interval $I(t)$. 
After localization of the integral, as we show, proving $\mathsf{LimInf}_p$ requires lower bound on the $p$-th moment of $\mathcal{Z}^{\mathbf{nw}}_{p,\epsilon}(t)e^{Y_{p,\epsilon}(t)}$ and $\int_{I(t)}e^{f_t(x)}dx$. Proposition~\ref{prop:liminf} and~\ref{prop:liminf2} will provide such lower bound. In what follows, we first state those propositions; prove $\mathsf{LimInf}_p$ and then, proceed to prove those ensuing propositions.  
\begin{prop}\label{prop:liminf}
We have
$\liminf_{\epsilon \to 0} \liminf_{t \to \infty} \tfrac{1}{t}\log \mathbb{E}[(\mathcal{Z}^{\mathbf{nw}}_{p,\epsilon}(t))^p e^{p Y_{p,\epsilon}(t)}] \geq \frac{p^3 - p}{24}.$
\end{prop}

\begin{prop}\label{prop:liminf2}
We have 
\begin{align}\label{eq:liminf 5}
\liminf_{\epsilon \to 0}\liminf_{t \to \infty} \frac{1}{t} \log\Big( e^{-\frac{p x_{p-\epsilon,\gamma} (t)^2}{2t}} \mathbb{E}\Big[\Big(\int_{I(t)} e^{f_t (x)} dx\Big)^p\Big]\Big) \geq g(p).
\end{align}
\end{prop}
\begin{proof}[\textbf{Proof of $\mathsf{LimInf}_p$:}] Due to Proposition \ref{prop:convolution}, it suffices to show the liminf of $t^{-1}\log\mathbb{E}[(\int^{\infty}_{-\infty}\mathcal{Z}^{\mathbf{nw}}(t,x)e^{f_t(x)})^{p}]$ as $t\to \infty$ is bounded below by $(p^3-p)/24+g(p)$. Since $\mathcal{Z}^{\mathbf{nw}}(t, x)$ and the exponential of $f_t (x)$ are both almost surely non-negative, $\int^{\infty}_{-\infty}\mathcal{Z}^{\mathbf{nw}}(t,x)e^{f_t(x)}dx $ is lower bounded by the integral of $ \mathcal{Z}^{\mathbf{nw}}(t, x) e^{f_t (x)} $ over the interval $I(t)$. We claim and prove that there exists a constant $C=C(p)>0$ such that 
\begin{equation}\label{eq:LimInfLowBd}
\mathbb{E}\Big[\Big(\int_{I(t)}\mathcal{Z}^{\mathbf{nw}}(t, x) e^{f_t (x)} dx\Big)^{p}\Big] \geq C \mathbb{E}\Big[(\mathcal{Z}^{\mathbf{nw}}_{p,\epsilon}(t))^p e^{p Y_{p,\epsilon}(t)}\Big]\cdot e^{-\frac{p x_{p-\epsilon, w} (t)^2}{2t}} \mathbb{E}\bigg[\Big(\int_{I(t)} e^{f_t (x)} dx\Big)^p\bigg].
\end{equation}
By assuming this inequality, we first prove $\mathsf{LimInf}_p$. We take logarithm of both sides of \eqref{eq:LimInfLowBd}, divide them by $t$ and let $t\to \infty$. After these set of operations, the liminf of the right hand side as $t\to \infty$, $\epsilon \to 0$ will be bounded below by $(p^3-p)/24+g(p)$ via the inequalities in Proposition~\ref{prop:liminf} and~\ref{prop:liminf2}. From this, the desired inequality of $\mathsf{LimInf}_p$ follows since $\mathbb{E}[(\int^{\infty}_{-\infty}\mathcal{Z}^{\mathbf{nw}}(t,x)e^{f_t(x)} dx)^{p}]$ exceeds $\mathbb{E}[(\int_{I(t)}\mathcal{Z}^{\mathbf{nw}}(t, x) e^{f_t (x)} dx)^{p}]$. In the rest of the proof, we focus on showing \eqref{eq:LimInfLowBd}. We first derive it from the following inequality: there exists $C=C(p)>0$ such that for all $x\in I(t)$,
\begin{equation}\label{eq:InMedIneq}
\mathcal{Z}^{\mathbf{nw}} (t, x)\geq C \mathcal{Z}^{\mathbf{nw}}\big(t, x_{p-\epsilon,\gamma}(t)\big) e^{Y_{p,\epsilon}(t)} .
\end{equation}
Owing to this, $\int_{I(t)} \mathcal{Z}^{\mathbf{nw}}(t, x) e^{f_t (x)}dx$ can be bounded below by the product of $C \mathcal{Z}^{\mathbf{nw}}\big(t, x_{p-\epsilon,\gamma}(t)\big) e^{Y_{p,\epsilon}(t)}$ and $\int_{I(t)}e^{f_t(x)}dx$. This readily implies  
\begin{align}
\mathbb{E}\Big[ \Big(\int_{I(t)} \mathcal{Z}^{\mathbf{nw}}(t, x) e^{f_t (x)} dx\Big)^p\Big] \geq  C\, \mathbb{E}\Big[\mathcal{Z}^{\mathbf{nw}} \big(t, x_{p-\epsilon,\gamma} (t)\big)^p e^{p Y_{p,\epsilon}(t)}\Big] \mathbb{E}\Bigg[\Big(\int_{I(t)} e^{f_t (x)} dx\Big)^p\Bigg]\label{eq:liminf 1} .
\end{align}
From the above inequality, \eqref{eq:LimInfLowBd} follows by multiplying and dividing the right hand side of \eqref{eq:liminf 1} by $\exp(px^{2}_{p-\epsilon,\gamma}(t)/2t)$ and recalling that $\mathcal{Z}^{\mathbf{nw}} \big(t, x_{p-\epsilon,\gamma} (t)\big)^p \exp(px^{2}_{p-\epsilon,\gamma}(t)/2t)$ is equal to $(\mathcal{Z}^{\mathbf{nw}}_{p,\epsilon}(t))^p$ (which is defined in \eqref{eq:Z(t)defin}). It remains to show \eqref{eq:InMedIneq} which we show as follows.

Recall that $Y_{p,\epsilon}(t)$ is defined as an infimum of the right hand side of \eqref{eq:Y(t)defin} over $I(t)$. So for all $x\in I(t)$,
\begin{equation*}
\mathcal{H}^{\mathbf{nw}}(t, x)  \geq \mathcal{H}^{\mathbf{nw}}\big(t, x_{p-\epsilon,\gamma} (t)\big) + Y_{p,\epsilon}(t)- \frac{\big(x -x_{p-\epsilon,\gamma} (t)\big) x_{p-\epsilon,\gamma} (t)}{t} - \frac{\nu \big(x - x_{p-\epsilon,\gamma} (t)\big)^2}{2}.
\end{equation*}
Taking the exponential on the both sides and recalling $\mathcal{Z}^{\mathbf{nw}} (t, x) = e^{\mathcal{H}^{\mathbf{nw}}(t, x)}$, we get
\begin{align*}
\mathcal{Z}^{\mathbf{nw}} (t, x) &\geq \mathcal{Z}^{\mathbf{nw}}\big(t, x_{p-\epsilon,\gamma} (t)\big) e^{Y_{p,\epsilon}(t)} e^{-\frac{\nu (x - x_{p-\epsilon,\gamma} (t))^2}{2}} e^{-\frac{(x- x_{p-\epsilon,\gamma} (t)) x_{p-\epsilon,\gamma} (t)}{t}}.
\end{align*}
By Lemma \ref{lem:condition}, for any fixed $p\in \mathbb{R}_{>0}$, there exists $C^{\prime}=C^{\prime}(p)>$ such that for all $t$ and $0 < \epsilon < \frac{p}{2}$, $|x_{p-\epsilon,\gamma} (t)| \leq C^{\prime}t$. Invoking this bound on the absolute value of $x_{p-\epsilon,\gamma} (t)$, we may lower bound the infimum value of $$\exp(-\nu (x-x_{p-\epsilon,\gamma}(t))^2/2-(x-x_{p-\epsilon,\gamma}(t))x_{p-\epsilon,\gamma}(t)/t)$$ as $x$ varies in $I(t)$ in the right hand side of the above display by some constant $C=C(p)>0$ (recall $I(t) = [\theta_{n(t)}, \theta_{n(t)+1}]$, whose length is no bigger than $1$). This yields \eqref{eq:InMedIneq} and hence, completes the proof of $\mathsf{LimInf}_p$.

\end{proof}

%
%

\begin{proof}[Proof of Proposition~\ref{prop:liminf}]
Our main goal is to show there exists $C=C(p,\epsilon)>0$ such that 
\begin{equation}
\mathbb{E}\Big[(\mathcal{Z}^{\mathbf{nw}}_{p,\epsilon}(t))^p e^{p Y_{p,\epsilon}(t)}\Big]\geq C  \Big(\mathbb{E}\big[(\mathcal{Z}^{\mathbf{nw}}_{p,\epsilon}(t))^{p-\epsilon}\big]\Big)^{\frac{p}{p-\epsilon}} .\label{eq:LowBdInEq}
\end{equation}
Before proceeding to the proof of the above inequality, we demonstrate how this implies Proposition~\ref{prop:liminf}. Taking the logarithm of both sides of \eqref{eq:LowBdInEq}, then dividing them by $t$ and letting $t \to \infty$ yields that 
\begin{equation}
\liminf_{t \to \infty} \frac{1}{t} \log\mathbb{E}\Big[(\mathcal{Z}^{\mathbf{nw}}_{p,\epsilon}(t))^p e^{p Y_{p,\epsilon}(t)}\Big] \geq \liminf_{t \to \infty} \frac{1}{t} \log \Big(\mathbb{E}\Big[(\mathcal{Z}^{\mathbf{nw}}_{p,\epsilon}(t))^{p-\epsilon}\Big]\Big)^{\frac{p}{p-\epsilon}} =
\frac{p}{p-\epsilon} \cdot \frac{(p-\epsilon)^3-(p-\epsilon)}{24},
\end{equation}
To see the last equality, we first note that $\mathcal{Z}^{\mathbf{nw}}_{p,\epsilon}(t)$ is same in distribution with $\mathcal{Z}^{\textbf{nw}}(t, 0)$ by Proposition \ref{prop:stationarity}. Combining this with Proposition \ref{prop:DT19thm} shows that the limit of $t^{-1}\log\mathbb{E}[(\mathcal{Z}^{\mathbf{nw}}_{p,\epsilon}(t))^{p-\epsilon}]$ is equal to $((p-\epsilon)^3-(p-\epsilon))/24$ as $t$ goes to $\infty$. As a consequence, we get the above equality. Letting $\epsilon \to 0$ in the above display, we obtain the desired result of Proposition~\ref{prop:liminf}. Thus, completing the proof of Proposition~\ref{prop:liminf} boils down to showing \eqref{eq:LowBdInEq} which we prove as follows.

 
 We write $(\mathcal{Z}^{\mathbf{nw}}_{p,\epsilon}(t))^{p-\epsilon}$ as a product of $X:=(\mathcal{Z}^{\mathbf{nw}}_{p,\epsilon}(t))^{p-\epsilon} e^{(p-\epsilon) Y_{p,\epsilon}(t)}$ and  $W:=e^{-(p-\epsilon) Y_{p,\epsilon}(t)}$. Applying the H\"{o}lder's inequality, we have $\mathbb{E}[XW]\leq \mathbb{E}[X^{p/(p-\epsilon)}]^{(p-\epsilon)/p}\mathbb{E}[W^{p/\epsilon}]^{\epsilon/p}$. Multiplying both sides of this inequality by $\mathbb{E}[W^{p/\epsilon}]^{-\epsilon/p}$ and raising both sides to the power $p/(p-\epsilon)$ yields 
\begin{equation}\label{eq:liminf 6}
\mathbb{E}\Big[(\mathcal{Z}^{\mathbf{nw}}_{p,\epsilon}(t))^{p} e^{p Y_{p,\epsilon}(t)}\Big]  \geq \Big(\mathbb{E}\Big[(\mathcal{Z}^{\mathbf{nw}}_{p,\epsilon}(t))^{p-\epsilon}\Big]\Big)^{\frac{p}{p-\epsilon}} \Big(\mathbb{E}\Big[e^{-\frac{p(p-\epsilon)}{\epsilon}Y_{p,\epsilon}(t)}\Big]\Big)^{-\frac{\epsilon}{p-\epsilon}}.   
\end{equation}
From the above inequality, \eqref{eq:LowBdInEq} follows once we show that $\mathbb{E}[\exp(-p (p-\epsilon) Y_{p,\epsilon}(t)/\epsilon)]$ is uniformly upper bounded by a constant $C^{\prime}=C^{\prime}(p,\epsilon)$ for all $t > 1$. This will be shown hereafter. For proving this bound, our main tools are the spatial stationarity $\mathcal{H}^{\mathbf{nw}}(t,x)+x^2/2t$ and the tail bounds of Proposition~\ref{prop: CGH 4.4}. By expressing $(x- x_{p-\epsilon,\gamma}(t)) x_{p-\epsilon,\gamma} (t)$ in the definition of $Y_{p,\epsilon}(t)$ as $2^{-1}\big(x^2-(x_{p-\epsilon,\gamma} (t))^2-(x - x_{p-\epsilon,\gamma}(t))^2\big)$, we may rewrite $Y_{p,\epsilon}(t)$ as
\begin{align*}
Y_{p,\epsilon}(t) &= \inf_{x \in I(t)} \Big(\mathcal{H}^{\mathbf{nw}} (t, x) + \frac{x^2}{2t} - \mathcal{H}^{\mathbf{nw}} \big(t, x_{p-\epsilon,\gamma} (t)\big) - \frac{x_{p-\epsilon,\gamma} (t)^2}{2t} - \frac{\big(x-x_{p-\epsilon,\gamma} (t)\big)^2}{2t} + \frac{\nu \big(x- x_{p-\epsilon,\gamma} (t)\big)^2}{2}\Big).
\end{align*}
Using stationarity of $\mathcal{H}^{\mathbf{nw}}(t, x) + \frac{x^2}{2t}$ in Proposition \ref{prop:stationarity}, we can shift the spatial variable $x$ of the above display to the left by $x_{p-\epsilon,\gamma} (t)$ and obtain the distributional identity  
\begin{equation}\label{eq:YDist}
Y_{p,\epsilon}(t) \overset{d}{=} \inf_{x \in I^{0}(t)} \Big(\mathcal{H}^{\mathbf{nw}} (t, x) + \frac{x^2}{2t} - \mathcal{H}^{\mathbf{nw}} (t, 0) - \frac{x^2}{2t} + \frac{\nu x^2}{2}\Big) = \inf_{x \in I^{0}(t)} \Big(\mathcal{H}^{\mathbf{nw}} (t, x) - \mathcal{H}^{\mathbf{nw}} (t, 0) + \frac{\nu x^2}{2}\Big),
\end{equation}
where $I^{0}(t):=[\theta_{n(t)}-x_{p-\epsilon,\gamma}(t), \theta_{n(t)+1}-x_{p-\epsilon,\gamma}(t)] \subseteq [-1, 1]$. 
Recall that \eqref{eq:increment lower tail} of Proposition~\ref{prop: CGH 4.4} provides a lower tail bound of the random variable $\inf_{x\in [0,1]}\{\mathcal{H}^{\mathbf{nw}}(t,x)-\mathcal{H}^{\mathbf{nw}}(t,0)+\frac{\nu x^2}{2t}\}$. Since the law of the process $\{\mathcal{H}^{\mathbf{nw}}(t,x)-\mathcal{H}^{\mathbf{nw}}(t,0)+\frac{\nu x^2}{2t}: x\in [0,1]\}$ is same as $\{\mathcal{H}^{\mathbf{nw}}(t,x)-\mathcal{H}^{\mathbf{nw}}(t,0)+\frac{\nu x^2}{2t}:x\in [0,-1]\}$, there exist $c=c(\delta,\nu)>0$, $s_0=s_0(\delta,\nu)>0$ such that for all $s>s_0$ 
$$\mathbb{P}\big(\inf_{x\in [-1,0]}\{\mathcal{H}^{\mathbf{nw}}(t,x)-\mathcal{H}^{\mathbf{nw}}(t,0)+\frac{\nu x^2}{2t}\}\leq -s\big)\leq \exp(-cs^{\frac{9}{8}-\delta}\big).$$
  Owing to the lower tail bound of $\inf_{x\in [0,1]}\{\mathcal{H}^{\mathbf{nw}}(t,x)-\mathcal{H}^{\mathbf{nw}}(t,0)+\frac{\nu x^2}{2t}\}$ and $\inf_{x\in [-1,0]}\{\mathcal{H}^{\mathbf{nw}}(t,x)-\mathcal{H}^{\mathbf{nw}}(t,0)+\frac{\nu x^2}{2t}\}$ and the distributional identity \eqref{eq:YDist}, for any $\delta \in (0,1)$, there exist $c=c(\delta,\nu)>0$ and $s_0=s_0(\delta,\nu)$ such that $\mathbb{P}(Y_{p,\epsilon}(t)\leq -s)\leq \exp(-cs^{9/8-\delta})$. We chose $\delta =\frac{1}{17}$. It is straightforward to see that $\frac{9}{8}-\frac{1}{17}>1+\frac{1}{17}$. As a consequence, we may write 
\begin{align}
\mathbb{E}[e^{-\frac{p (p-\epsilon)}{\epsilon} Y_{p,\epsilon}(t)}]\leq e^{\frac{p(p-\epsilon)}{\epsilon}s_0}\mathbb{P}(Y_{p,\epsilon}(t)\geq -s_0) +\frac{p(p-\epsilon)}{\epsilon}\int^{\infty}_{s_0}e^{\frac{p(p-\epsilon)}{\epsilon}s}e^{-cs^{1+\frac{1}{17}}}ds.
\end{align}  
The integral on the right hand side of the above inequality is finite and its value is equal to some constant $C^{\prime\prime}=C^{\prime\prime}(\nu,p,\epsilon)>0$. This demonstrates why $\mathbb{E}[\exp(-p (p-\epsilon) Y_{p,\epsilon}(t)/\epsilon)]$ is bounded by some constant which only depends on $\nu,p$ and $\epsilon$. Substituting this bound into \eqref{eq:liminf 6} yields \eqref{eq:LowBdInEq}. This completes the proof of Proposition~\ref{prop:liminf}.  

\end{proof}

\begin{proof}[Proof of Proposition~\ref{prop:liminf2}] To prove \eqref{eq:liminf 5}, we show the following inequality: there exists constant $C=C(p,\epsilon)>0$ such that 
\begin{equation}\label{eq:liminf2Ineq}
\mathbb{E}\Big[\Big(\int_{I(t)} e^{f_t (x)} dx\Big)^p\Big] \geq C|I(t)|^{p} \Big(\mathbb{E}\Big[e^{(p-\epsilon) f_t (x_{p-\epsilon,\gamma} (t)) }\Big]\Big)^{\frac{p}{p-\epsilon}},
\end{equation}
where $|I(t)|$ is the length of the interval $I(t) = [\theta_{n(t)},\theta_{n(t)+1}]$.
Let us explain why the above inequality is sufficient for proving \eqref{eq:liminf 5}. Owing to \eqref{eq:liminf2Ineq}, we may write 
\begin{align*}
\log\Big( e^{-\frac{p x^2_{p-\epsilon,\gamma} (t)}{2t}} \mathbb{E}\Big[\Big(\int_{I(t)} e^{f_t (x)} dx\Big)^p\Big]\Big) 
\geq -\frac{p x^2_{p-\epsilon,\gamma} (t)}{2t} + \frac{p}{p-\epsilon}\log \mathbb{E}\Big[e^{(p-\epsilon) f_t (x_{p-\epsilon,\gamma} (t))}\Big] +p\log |I(t)|+ \log C. 
\end{align*}
Recall that $1\geq |\theta_n - \theta_{n+1}| \geq \max\{1, c|n|^{-\beta}\}$ for some $c>0$, $\beta \in (0,1)$ and all $n \in \mathbb{Z}$ by the \emph{pseudo-stationarity} condition of Definition \ref{Def:HypDef} and $|\theta_{n(t)}|\leq Ct$ for some constant $C=C(p,\epsilon)>0$ by Lemma~\ref{lem:condition}-$(ii)$. From the inequality $|\theta_n - \theta_{n+1}| \geq \max\{1, c|n|^{-\beta}\}$, we get $|\theta_{n}|\geq cn^{1-\beta}$ for some constant $c= c(\beta)>0$. Combining this with the upper bound $|\theta_{n(t)}|\leq Ct$ yields $n(t) \leq |C t|^{1/(1-\beta)}$ and hence, shows  $|I(t)|\geq |Ct|^{-\beta/(1-\beta)}$. Conjugating this last inequality with the upper bound $|I(t)|\leq 1$ implies that $t^{-1}\log |I(t)|$ converges to $0$ as $t\to \infty$. Now, dividing both sides of the above display by $t$ and letting $t \to \infty$ followed by $\epsilon\to 0$, $\gamma\to 0$ shows \eqref{eq:liminf 5} if the following inequality is satisfied 
\begin{align}
\liminf_{\epsilon\to 0}\liminf_{t \to \infty} \frac{1}{t} \Big( -\frac{p x^2_{p-\epsilon,\gamma} (t)}{2t} + \frac{p}{p-\epsilon}\log \mathbb{E}\Big[e^{(p-\epsilon) f_t (x_{p-\epsilon,\gamma} (t))}\Big]\Big)\geq g(p).\label{eq:gplow}
\end{align}
We prove this inequality as follows. By taking a factor $p/(p-\epsilon)$ out of the parantheses of the left hand side of the above display and recalling the definition of $g(\cdot)$ from \eqref{eq:Gp} of Definition~\ref{Def:HypDef}, we may write $$\text{l.h.s. of \eqref{eq:gplow}}\geq \liminf_{\gamma\to 0}\liminf_{\epsilon\to 0}\frac{p}{p-\epsilon}g(p-\epsilon).$$ 
Since $g$ is a convex function, $g$ is continuous at $p$. This shows that $\liminf_{\epsilon\to 0}pg(p-\epsilon)/(p-\epsilon)$ is equal to $g(p)$ and indeed, \eqref{eq:gplow} holds. Consequently, we get \eqref{eq:liminf 5} modulo \eqref{eq:liminf2Ineq}. The rest of the proof will show \eqref{eq:liminf2Ineq}.


Recall the definition of $\mathrm{TV}_{f_t}(\cdot)$ from the proof of Proposition~\ref{prop:InitControl}.  
 Since $\mathrm{TV}_{f_t}(n(t))$ is the supremum of $|f_t (x) - f_t (\theta_{n(t)})|$ over $x \in I(t)$, we know that $f_t (x) \geq f_t (x_{p-\epsilon,\gamma}(t)) - 2\mathrm{TV}_{f_t}(n(t))$ for all $x \in I(t)$. Taking the exponential on both sides of this inequality and then integrating on $I(t)$ shows $\int_{I(t)}\exp(f_t(x))dx\geq |I(t)|\exp(f_t(x_{p-\epsilon,\gamma}(t))-2\mathrm{TV}_{f_t}(n(t)))$ which after raising to $p$-th power and taking expectation yields
\begin{equation}\label{eq:liminf 8}
\mathbb{E}\Big[\Big(\int_{I(t)} e^{f_t (x)} dx\Big)^p\Big] \geq |I(t)|^{p}\mathbb{E}\Big[ e^{p f_t (x_{p-\epsilon,\gamma} (t))} e^{-2p \mathrm{TV}_{f_t}(n(t))}\Big].
\end{equation}
It will suffice to show that the right hand side of the above display is bounded below by a constant multiple of $(\mathbb{E}[e^{(p-\epsilon) f_t (x_{p-\epsilon,\gamma} (t)) }])^{p/(p-\epsilon)}$. To get this lower bound, we write $e^{(p-\epsilon) f_t (x_{p-\epsilon,\gamma} (t))}$ as a product of two random variables $\mathcal{X}:=e^{(p-\epsilon) f_t (x_{p-\epsilon,\gamma} (t))} e^{-(p-\epsilon) \mathrm{TV}_{f_t} (n(t))}$ and $\mathcal{W}:=e^{(p-\epsilon) \mathrm{TV}_{f_t}(n(t))}$. Using the H\"{o}lder inequality, we get $\mathbb{E}[\mathcal{X}\mathcal{W}]\leq (\mathbb{E}[\mathcal{X}^{p/(p-\epsilon)}])^{(p-\epsilon)/p}(\mathbb{E}[\mathcal{W}^{p/\epsilon}])^{\epsilon/p}$. Multiplying both sides of this inequality by $(\mathbb{E}[\mathcal{W}^{p/\epsilon}])^{-\epsilon/p}$ and raising both sides to the power $p/(p-\epsilon)$ results in
\begin{equation*}
\mathbb{E}\Big[e^{p f_t (x_{p-\epsilon,\gamma} (t))} e^{-p \mathrm{TV}_{f_t} (n(t))}\Big] \geq \Big(\mathbb{E}\Big[e^{(p-\epsilon) f_t (x_{p-\epsilon,\gamma} (t))}\Big]\Big)^{\frac{p-\epsilon}{p}}  \Big(\mathbb{E}\Big[e^{\frac{p(p-\epsilon)}{\epsilon} \mathrm{TV}_{f_t} (n(t))}\Big]\Big)^{-\frac{\epsilon}{p-\epsilon}}.
\end{equation*}
By the super-exponential tail bounds for $\mathrm{TV}_{f_t}(\cdot)$ specified in \eqref{eq:ModCont} of Definition~\ref{Def:HypDef}, we know that $\mathbb{E}[\exp(p (p-\epsilon)\mathrm{TV}_{f_t} (n(t))/\epsilon)]$ is upper bounded by a constant. Combining this observation with the inequality of the above display yields 
\[\mathbb{E}\Big[e^{p f_t (x_{p-\epsilon,\gamma} (t))} e^{-2p\mathrm{TV}_{f_t}(n(t))}\Big] \geq C \Big(\mathbb{E}\Big[e^{(p-\epsilon) f_t (x_{p-\epsilon,\gamma} (t)) }\Big]\Big)^{\frac{p}{p-\epsilon}},\]
for some $C>0$. Substituting this into the right hand side of \eqref{eq:liminf 8} gives \eqref{eq:liminf2Ineq}. This completes the proof.

\end{proof}

\subsection{Proof of \eqref{eq:ldp}}\label{sec:ProofLDP}
We take $X(t) = \mathcal{H}^{f_t} (t, 0) + \frac{t}{24}$, by Theorem \ref{thm:main} part (a), we see that $$\lim_{t \to \infty} \frac{1}{t} \log \mathbb{E}\big[\exp(p X(t))\big] = \frac{p^3}{24} + g(p).$$ 
We conclude \eqref{eq:ldp} via applying Proposition \ref{prop:intermittldp}. It suffices to verify $h(p) := g(p) + \frac{p^3}{24}$ indeed satisfies the condition in Proposition \ref{prop:intermittldp}. By Lemma \ref{lem:condition}, $g(p)$ is convex, since $g \in C^1(\mathbb{R}_{>0})$ as we assume, thus $g'(p)$ is increasing. Consequently, $h'(p) = g'(p) + \frac{p^2}{8}$ is continuous and strictly increasing on $(0, \infty)$. Moreover, $\lim_{p \to 0} h'(p) = \lim_{p \to 0} g'(p) + \frac{p^2}{8} = \zeta$ and 
\begin{equation*}
\lim_{p \to \infty} h'(p) = \lim_{p \to \infty} g'(p) + \frac{p^2}{8} = \infty.
\end{equation*}
This implies that $h'(p)$ is a continuous bijection from $(\zeta, \infty)$ to $(0, \infty)$, so it satisfies the condition in Proposition \ref{prop:intermittldp}. Applying this proposition completes the proof of \eqref{eq:ldp}.

\section{Proof of Theorem~\ref{cor:detinitial} \& Theorem~\ref{cor:bminitial}}\label{sec:Cor}

\subsection{Proof of Theorem~\ref{cor:detinitial}}\label{sec:cordet}
Recall that the deterministic initial profile $f_0$ satisfies the growth condition of \eqref{eq:detgrowth} and finite oscillation property as stated in Theorem~\ref{cor:detinitial}. Using $f_0$, we construct the set of functions $(g,\{f_t\}_{t>0})$ with $g\equiv 0$ and $f_t:=f_0$ for all $t>0$.  We claim that $(g,\{f_t\}_{t>0})$ belongs to the class $\mathbf{Hyp}$ which is characterized by three conditions, namely, (1) growth and lower bound conditions, (2) pseudo-stationarity condition and, (3) coherence conditions (see Section~\ref{sec:proofidea}). Modulo this claim, by Theorem~\ref{thm:main}, we have \eqref{eq:upperbound}. Furthermore, since $g\in C^{1}(\mathbb{R}_{>0})$ with $\zeta= \lim_{p\to 0}g^{\prime}(p)=0$, by \eqref{eq:ldp}, we get 
\begin{equation}
\lim_{t \to \infty} \frac{1}{t} \log \mathbb{P}\Big(\mathcal{H}^{f_t}(t, 0) + \frac{t}{24} > ts\Big) = -\sup_{s > 0} \big(ps - p^3/24\big) = -\frac{4 \sqrt{2}}{3} s^{\frac{3}{2}}.
 \end{equation}
This shows \eqref{eq:detLDP}. To complete the proof of Theorem~\ref{cor:detinitial}, it suffices to verify our claim that $(g,\{f_t\}_{t>0})$ belongs to the class $\mathbf{Hyp}$, i.e., $(g,\{f_t\}_{t>0})$ has to satisfy \eqref{eq:Gp},\eqref{eq:condition 3}, \eqref{eq:growth condition}, \eqref{eq:lower bound condition} and \eqref{eq:ModCont}. Note that \eqref{eq:growth condition} and \eqref{eq:lower bound condition} follow immediately from the property $(i)$ of $f_t$ which says that there exist $\delta\in (0,1)$ and constant $C>0$ such that $|f_t(x)|\leq C(1+|x|^{\delta})$ for all $x\in \mathbb{R}$. In what follows, we successively prove \eqref{eq:Gp}, \eqref{eq:condition 3} and \eqref{eq:ModCont} for $(g,\{f_t\}_{t>0})$. 

\textsc{Proof of \eqref{eq:Gp}:} Since $\{f_t(\cdot)\}_{t\geq 0}$ is a sequence of deterministic initial data, the following limit 
\begin{equation}\label{eq:gpproof}
\lim_{t \to \infty} \frac{1}{t} \sup_{x \in \mathbb{R}} \Big(-\frac{px^2}{2t} + \log e^{p f_t (x)}\Big) = 0
\end{equation}
will show \eqref{eq:Gp}. Our main objective is to prove \eqref{eq:gpproof}.
For all large $t>0$, owing to the growth condition $|f_t(x)| \leq C(1+|x|^{\delta})+ \frac{\alpha x^2}{2t}$ for some constant $C>0$ and $\delta,\alpha\in (0,1)$,  
\begin{equation*}
\underbrace{\sup_{x \in \mathbb{R}} \Big\{-\frac{px^2}{2t} - Cp(1 + |x|^{\delta})-\frac{p\alpha x^2}{2t}\Big\}}_{\mathsf{Sup}^{(1)}_t} \leq \sup_{x \in \mathbb{R}} \Big\{-\frac{px^2}{2t} + \log e^{p f_t(x)}\Big\} \leq \underbrace{\sup_{x \in \mathbb{R}} \Big\{-\frac{px^2}{2t} + Cp(1 + |x|^{\delta})+\frac{p\alpha x^2}{2t}\Big\}}_{\mathsf{Sup}^{(2)}_{t}}.
\end{equation*}
In order to prove \eqref{eq:gpproof}, it suffices to show $t^{-1}\mathsf{Sup}^{(1)}_t$ and $t^{-1}\mathsf{Sup}^{(2)}_t$ converge to $0$ as $t$ tends to $\infty$. We only show $t^{-1}\mathsf{Sup}^{(2)}_t\to 0$ as $t\to \infty$. The other convergence follows verbatim. We rewrite $\mathsf{Sup}^{(2)}_t$ as $ Cp + \sup_{x \in \mathbb{R}} \{-p(1-\alpha)x^2/2t + Cp |x|^{\delta}\}$.
We do a change of variable $x \to t^{1/(2-\delta)} x$ in this new form of $\mathsf{Sup}^{(2)}_t$. As a consequence,  we can further rewrite $\mathsf{Sup}^{(2)}_t$ as $Cp+pt^{\delta/ (2-\delta)} \sup_{x \in \mathbb{R}} \{-(1-\alpha)x^2/2 + C|x|^{\delta}\}$. Note that the function $\phi(x)= -(1-\alpha)x^2/2+C|x|^\delta$ satisfies $\phi(0)=0$ and $\phi(+\infty)=\phi(-\infty)=-\infty$. Thus, the supremum value of $\phi(x)$ as $x$ varies in $\mathbb{R}$ is finite. This shows we may upper bound $\mathsf{Sup}^{(2)}_t$ by $Cp+C^{\prime}pt^{\delta/(2-\delta)}$ for some constant $C^{\prime}>0$ and lower bound it by $Cp$. These upper and lower bound when divided by $t$ with letting $t\to\infty$ converge to $0$. This proves the claim that $t^{-1}\mathsf{Sup}^{(2)}_t\to 0$ as $t\to \infty$ and hence, shows \eqref{eq:gpproof}.
\smallskip
\\ 

\textsc{Proof of \eqref{eq:condition 3}:} Note that \eqref{eq:condition 3} will follow if the following limit holds 
\begin{equation}\label{eq:gpCond3}
\lim_{t \to \infty} \frac{1}{t} \log \Big(\int e^{-\frac{p(1-\epsilon) x^2}{2t}} \cdot e^{p (1+\epsilon) f_t (x)} dx\Big) = 0
\end{equation}
for all small $\epsilon>0$. Throughout the rest of the proof, we show \eqref{eq:gpCond3}. Since there exist $C>0$, $\delta,\alpha\in(0,1)$ such that $|f_t(x)| \leq C (1+|x|^{\delta})+ \frac{\alpha x^2}{2t}$ for all large $t>0$, we may write 
\begin{align}
\int e^{-\frac{p((1-\epsilon)+\alpha(1+\epsilon)) x^2}{2t} - Cp(1+\epsilon) |x|^{\delta}}   dx \leq \int e^{-\frac{p(1-\epsilon) x^2}{2t}} e^{p(1+\epsilon) f_t (x)} dx &\leq  e^{Cp (1+\epsilon)}  \int e^{-\frac{p((1-\epsilon)-\alpha(1+\epsilon)) x^2}{2t} + Cp(1+\epsilon) |x|^{\delta}}   dx\label{eq:IntEq}
\end{align}
We choose $\epsilon$ small such that $(1-\epsilon)-\alpha(1+\epsilon)>0$. For proving \eqref{eq:gpCond3}, one needs to show that the logarithm of the left and right hand side of \eqref{eq:IntEq} when divided by $t$ with $t\to \infty$ converge to $0$. We only show this for the right hand side and the other convergence follows from similar argument. For convenience, we denote the right hand of \eqref{eq:IntEq} by $\mathsf{RHS}_{t}$. By a change of variable $x \to t^{1/ (2-\delta)} x$ inside the integral of $\mathsf{RHS}_{t}$, we may write  
\begin{align}\label{eq:RHS}
\mathsf{RHS}_{t}= e^{Cp (1+\epsilon)}t^{1/(2-\delta)} \int  e^{t^{\frac{\delta}{2 -\delta}} \big(-p((1-\epsilon)-\alpha(1+\epsilon))x^2 / 2 + Cp(1+\epsilon) |x|^{\delta}\big)} dx.
\end{align}
By splitting the domain of the above integral into two parts $\{x:|x|\leq 1\}$ and $\{x:|x|>1\}$, we write $\mathsf{RHS}_{t}$ as sum of $\exp(Cp (1+\epsilon))t^{1/(2-\delta)}\mathcal{A}_1$ and $\exp(Cp (1+\epsilon))t^{1/(2-\delta)}\mathcal{A}_2$ where $\mathcal{A}_1$ and $\mathcal{A}_2$ denote the integral in \eqref{eq:RHS} computed over the region $\{x:|x|\leq 1\}$ and $\{x:|x|>1\}$ respectively. To show $t^{-1}\log \mathsf{RHS}_{t}\to 0$ as $t\to \infty$, we first find upper bound to $\mathcal{A}_1$ and $\mathcal{A}_2$. Since $-p\big((1-\epsilon)-\alpha(1+\epsilon)\big)x^2/2 + Cp|x|^{\delta}$ is bounded by some constant $C^{\prime}=C^{\prime}(p,\epsilon,\alpha)>0$ for all $|x|\leq 1$ and all large $t$, we can bound $\mathcal{A}_1$ by $\exp(C^{\prime}t^{\delta/(2-\delta)})$. By using the inequality $|x|^{\delta}< |x|$ for all $|x|>1$ (holds since $\delta <1$), we may write 
\begin{equation*}
\mathcal{A}_2 \leq \int_{|x|>1} e^{t^{\frac{\delta}{2-\delta}} \big(-p((1-\epsilon)-\alpha(1+\epsilon))x^2/2 + Cp(1+\epsilon) |x|\big)} dx\leq \int   e^{t^{\frac{\delta}{2-\delta}} \big((-p(1-\epsilon)-\alpha(1+\epsilon))x^2/2 + Cp(1+\epsilon) |x|\big)} dx,
\end{equation*}
where the last inequality is obtained by leveraging the positivity of the integrand. Note that the integral on the right hand side of the above display is a Gaussian integral. It is straightforward to see that this Gaussian integral can be bounded above by $C_1t^{-\delta/(2(2-\delta))}\exp(C_2t^{\delta/(2-\delta)})$ for some $C_1=C_1(p,\alpha,\epsilon)>0$ and $C_2=C_2(p,\alpha,\epsilon)>0$. Combining the upper bounds on $\mathcal{A}_1$ and $\mathcal{A}_2$ and substituting those into the right hand side of \eqref{eq:RHS} yields 
\begin{align}
\mathsf{RHS}_{t}\leq e^{Cp (1+\epsilon)}t^{1/(2-\delta)}\Big(e^{C^{\prime}t^{\delta/(2-\delta)}} + C_1t^{-\delta/(2(2-\delta))}e^{C_2t^{\delta/(2-\delta)}}\Big),
\end{align}
where $C^{\prime}, C_1, C_2$ are some positive constants depending on $p,\alpha$ and $\epsilon$. Taking logarithm on both sides, dividing them by $t$ and letting $t\to \infty$ shows $t^{-1}\log \mathsf{RHS}_{t}$ converge to $0$. This completes the proof of \eqref{eq:condition 3}.
\smallskip
\\

\textsc{Proof of \eqref{eq:ModCont}:} This will be proved using bounded local oscillation property of the initial data $f_0$ in Theorem~\ref{cor:detinitial}. Recall that $\sup_{I\subset \mathbb{R}, |I|\leq \alpha} \mathrm{Osc}_{I}(f_0)<\infty$ for some $\alpha>0$. Define $\theta_{n}:=(\alpha\wedge 1)n$ for all $n\in \mathbb{Z}$. Notice $\{\theta_n\}_{n\in \mathbb{Z}}$ is a bi-infinite sequence with $\theta_0=0$ and $(\alpha\wedge 1)\leq |\theta_n-\theta_{n+1}|\leq 1$ for all $n$. Furthermore, $\sup_{x\\in [\theta_n,\theta_{n+1}]}|f_0(x) - f_0(\theta_{n})|\leq s_0$ for some $s_0$ since $\sup_{I\subset \mathbb{R}, |I|\leq \alpha} \mathrm{Osc}_{I}(f_0)<\infty$. This implies the pseudo-stationarity condition is trivially satisfied for the sequence $\{f_t\}_{t>0}$ where $f_t= f_0$.

\subsection{Proof of Theorem~\ref{cor:bminitial}}\label{sec:corbm}
\red{For all $t>0$, define $f_t:\mathbb{R}\to \mathbb{R}$ as $f_t (x) := (\sigma_+ B(x) + a_+ x) \mathbbm{1}_{\{x \geq 0\}} - (-\sigma_- B(x) + a_- x) \mathbbm{1}_{\{x \leq 0\}}$ and define $g:(0,\infty)\to\mathbb{R}$ as 
\begin{equation}\label{eq:gp}
g(p)= \frac{p}{2} \max\big\{\frac{p\sigma_+^2}{2} + a_+, \frac{p\sigma_-^2}{2} + a_- ,0\big\}^2.
\end{equation}}
We claim and prove that $(g,\{f_t\}_{t\geq 0})$
belongs to the class $\mathbf{Hyp}$. For now, we assume this claim and show how this implies \eqref{eq:bmLyapunov} and \eqref{eq:bmldp}.

Note that \eqref{eq:bmLyapunov} follows immediately from \eqref{eq:intermittency} of Theorem~\ref{thm:main} since $(g,\{f_t\}_{t\geq 0})\in \mathbf{Hyp}$ by our assumption. We turn now to show \eqref{eq:bmldp}. \red{Recall that $a=\max\{a_{+},a_{-}\}$. One can readily verify $g \in C^1(\mathbb{R}_{>0})$, $\zeta = \lim_{p \to 0} g'(p) =  \frac{\max(a, 0)^2}{2}$. 
By \eqref{eq:ldp} of Theorem~\ref{thm:main}, when $s > \frac{a^2}{2}$, 
\begin{equation}\label{eq:ldp3}
\lim_{t \to \infty} \frac{1}{t} \log \mathbb{P}\Big(\mathcal{H}_t^{f_t} (0) + \frac{t}{24} \geq st\Big) = \sup_{p > 0} \Big\{ps - \frac{p^3}{24} - g(p)\Big\}.
\end{equation}
Since $\lim_{p \to 0} g(p) = 0$, a direction computation shows that 
\begin{equation}\label{eq:ldp2}
\lim_{s \to \frac{\max(a, 0)^2}{2}} \sup_{p > 0}\{ps - \frac{p^3}{24} - g(p)\} = 0.
\end{equation}
Note that the left hand side of \eqref{eq:ldp3} is decreasing in $s$ and is non-positive, hence \eqref{eq:ldp2} implies that when $s \leq \frac{a^2}{2}$, 
\begin{equation*}
\lim_{t \to \infty} \frac{1}{t} \log \mathbb{P}\Big(\mathcal{H}_t^{f_t} (0) + \frac{t}{24} \geq st\Big) = 0.
\end{equation*}
When $\sigma_+ = \sigma_- = 1$,  referring to \eqref{eq:gp}, $g(p) = \frac{p}{2}\max\{\frac{p}{2} + a, 0\}^2$. A direct computation for the right hand side of \eqref{eq:ldp3} yields \eqref{eq:bmldp1} and \eqref{eq:bmldp2}.} 
\bigskip
\\
We now turn to prove our claim $(g,\{f_t\}_{t\geq 0})\in \mathbf{Hyp}$. For this, we serially show that $(g,\{f_t\}_{t\geq 0})$ satisfies \eqref{eq:Gp}, \eqref{eq:condition 3}, \eqref{eq:growth condition}, \eqref{eq:lower bound condition} and \eqref{eq:ModCont}.
\smallskip
\\

\textsc{Proof of \eqref{eq:Gp}:} Since $\mathbb{E}[e^{pB(x)}]= \exp(p^2|x|/2)$ for any $x\in \mathbb{R}$,  \red{we have $\log \mathbb{E}[e^{p f_t (x)}] =  (\frac{p^2 \sigma_+^2 x}{2} + pa_+ x) \mathbbm{1}_{\{x \geq 0\}} + (\frac{p^2 \sigma_-^2 x}{2} - pa_- x )\mathbbm{1}_{\{x \leq 0\}}.$ By a direct computation, we get that the maximum value of $-px^2/2t+ \log \mathbb{E}[e^{p f_t (x)}]$ over $x \in \mathbb{R}$ is given by $\frac{pt}{2} \big(\max\{\frac{p \sigma_+}{2} + a_+,  \frac{p \sigma_-}{2} + a_-, 0\}\big)^2$. 
\begin{equation*}
\lim_{t \to \infty} \frac{1}{t} \sup_{x \in \mathbb{R}}\Big(-\frac{px^2}{2t} +  \log \mathbb{E}\Big[e^{p f_t (x)}\Big]\Big) = \frac{p}{2} \big(\max\{\frac{p \sigma_+^2}{2} + a_+,  \frac{p \sigma_-^2}{2} + a_-, 0\}\big)^2 = g(p).
\end{equation*}}
This verifies \eqref{eq:Gp}.

\textsc{Proof of \eqref{eq:condition 3}:} 
By using the inequality $ pa_+ x \mathbbm{1}_{\{x \geq 0\}} - pa_- x \mathbbm{1}_{\{x \leq 0\}}\leq pa|x|$ and the identity $\mathbb{E}[e^{pB(x)}]= e^{p^2|x|/2}$, \red{we get 
\begin{equation*}
\mathbb{E}\Big[e^{p(1+\epsilon) f_t (x)}\Big] =  e^{ \sigma_+^2 p^2 (1+\epsilon)^2 x/2 + p(1+\epsilon) a_+ x} \mathbbm{1}_{\{x > 0\}} + e^{-\sigma_-^2 p^2 (1+\epsilon)^2 x/2 - p(1+\epsilon) a_- x} \mathbbm{1}_{\{x < 0\}}.
\end{equation*} Owing to this inequality, we may write  
\begin{align}
\notag
&\int e^{-\frac{p(1-\epsilon) x^2}{2t}} \mathbb{E}\Big[e^{p(1+\epsilon) f_t (x)}\Big] dx 
\\
\label{eq:PxEq}
&\leq \int_0^\infty e^{-\frac{p(1-\epsilon) x^2}{2t}}  e^{\sigma_+^2 p^2 (1+\epsilon)^2 x/2 + p(1+\epsilon) a_+ x} dx + \int_{-\infty}^0 e^{-\frac{p(1-\epsilon) x^2}{2t}}  e^{-\sigma_-^2 p^2 (1+\epsilon)^2 x/2 - p(1+\epsilon) a_- x} dx.
\end{align}
We claim that 
\begin{align}
\label{eq:gp1}
&\limsup_{\e \to 0} \lim_{t \to \infty} \frac{1}{t}\log\int_0^\infty e^{-\frac{p(1-\epsilon) x^2}{2t}}  e^{\sigma_+^2 p^2 (1+\epsilon)^2 x/2 + p(1+\epsilon) a_+ x} dx \leq \frac{p}{2}\big(\max\big\{\frac{p \sigma_+^2}{2} + a_+, 0\big\}\big)^2, \\
\label{eq:gp2}
&\limsup_{\e \to 0} \lim_{t \to \infty} \int_{-\infty}^0 e^{-\frac{p(1-\epsilon) x^2}{2t}}  e^{-\sigma_-^2 p^2 (1+\epsilon)^2 x/2 - p(1+\epsilon) a_- x} dx \leq \frac{p}{2}\big(\max\big\{\frac{p \sigma_-^2}{2} + a_-, 0\big\}\big)^2.
\end{align}
Applying \eqref{eq:gp1} and \eqref{eq:gp2} to upper bound the right hand side of \eqref{eq:PxEq} yields \eqref{eq:condition 3}. 
\bigskip
\\
It remains to prove \eqref{eq:gp1} - \eqref{eq:gp2}. We only prove \eqref{eq:gp1}. The proof of \eqref{eq:gp2} follows in a similar way. 
For proving \eqref{eq:gp1}, we first consider the case $a_+< -\sigma_+^2 p/2$ and then, will move onto the case $a_+ \geq-\sigma_+^2 p/2$. 
For any $a_+ <0$ and $p>0$ satisfying $a_+ <-\sigma_+^2 p/2$, there exists $\epsilon_0=\epsilon_0(a,p)>0$ such that $$\frac{\sigma_+^2 p^2}{2}(1+\epsilon)^2+ a_+ p(1+\epsilon)<0,\quad \forall 0<\epsilon<\epsilon_0.$$   
Therefore, we can upper bound the integral on the left hand side of \eqref{eq:gp1} by $\int_0^\infty \exp(-p(1-\epsilon)x^2/2t)dx$ for all $0<\epsilon<\epsilon_0$. Since the limit of $t^{-1}\log(\int_0^\infty \exp(-p(1-\epsilon)x^2/2t)dx)$ is equal to $0$ as $t\to \infty$ for  all small $\epsilon>0$, we get 
\begin{equation}\label{eq:gp3}
\limsup_{\e \to 0} \lim_{t \to \infty} \frac{1}{t}\log\int_0^\infty e^{-\frac{p(1-\epsilon) x^2}{2t}}  e^{\sigma_+^2 p^2 (1+\epsilon)^2 x/2 + p(1+\epsilon) a_+ x} dx\leq 0, \quad \text{when } a_+ <-\frac{\sigma_+^2 p}{2}.
\end{equation}
Now, we turn to the case $a_+ \geq -\sigma_+^2 p/2$. The integral on the left hand side of \eqref{eq:gp1} can be upper bounded the Gaussian integrals $\int e^{-\phi(x)} dx$ 
where
\begin{equation*}
\phi(x)_:=p(1-\epsilon) x^2/(2t)-x p^2 (1+\epsilon)^2 \sigma_+^2/2 - xp(1+\epsilon) a_+.
\end{equation*}  
By a direct computation, one can show that $\int e^{-\phi(x)}dx $ 
is equal to $\sqrt{\frac{2\pi t}{p(1-\epsilon)}} \exp\Big(\frac{(\sigma_+^2 p^2 (1+\epsilon) /2 + p(1+\epsilon) a_+)^2 t}{2p (1-\epsilon)}\Big)$. With this exact formula, it is straightforward to check that 
\begin{equation}
\liminf_{\epsilon \to 0} \limsup_{t \to \infty} \frac{1}{t} \log\bigg(\sqrt{\frac{2\pi t}{p(1-\epsilon)}} \exp\Big(\frac{(\sigma_+^2 p^2 (1+\epsilon) /2 + p(1+\epsilon) a_+)^2 t}{2p (1-\epsilon)}\Big)\bigg) =  \frac{(\sigma_+^2 p^2/2 + p a_+)^2}{2p}. 
\label{eq:gPx}
\end{equation}
Combining \eqref{eq:gp3} - \eqref{eq:gPx} concludes the proof of \eqref{eq:gp1}.
This completes the proof of \eqref{eq:condition 3}.}
\smallskip 
\\
\textsc{Proof of \eqref{eq:growth condition} \& \eqref{eq:lower bound condition}:}
We know  $M_p^{f_t}(t,x)= \mathbb{E}\Big[e^{p f_t(x)}\Big] = e^{p^2 \red{\sigma^2_+}|x|/2 + a_{+} x}$ if $x>0$ and is equal to $e^{p^2 \red{\sigma^2_-} |x|/2 - a_{-} x}$ if $x\leq 0$. From this exact formula of $M^{f_t}_p(t,x)$, it is clear that the growth condition \eqref{eq:growth condition} holds for $\{f_t\}_{t>0}$. Since $f_t$ is same for all $t>0$, so \eqref{eq:lower bound condition} is true.
\smallskip 
\\

\textsc{Proof of \eqref{eq:ModCont}:} From the definition of $f_t(x)$, we know 
\red{\begin{align}
\frac{|f_t (x) - f_t (y)|}{|x-y|^{\frac{1}{2}}}= \begin{cases} 
\frac{|\red{\sigma_+} (B(x)-B(y)) + a_{+}(x-y)|}{|x - y|^\frac{1}{2}}  & x,y \geq 0,\\
\frac{|\red{\sigma_-}(B(x)-B(y)) - a_{-}(x-y)|}{|x - y|^\frac{1}{2}}  & x,y\leq 0,
\end{cases} \label{eq:fxy}
\end{align}}
\red{Let $\sigma = \max(\sigma_+, \sigma_-)$.} From the above relations, we intend to show that the following inequality 
\begin{equation}
\label{eq:bmModCont}
\frac{|f_t (x) - f_t (y)|}{|x-y|^{\frac{1}{2}}} \leq  \frac{\sigma|B(x) - B(y)|}{|x-y|^\frac{1}{2}} + \max(|a_+|, |a_-|)
\end{equation}
holds for all $x,y\in \mathbb{R}$ such that $|x-y|\leq 1$ and $xy \geq 0$. 
Now, we show how \eqref{eq:bmModCont} implies \eqref{eq:ModCont}. We define $\theta_{n}=n$ for all $n\in \mathbb{Z}$. Fix any $n\in \mathbb{Z}$. By \eqref{eq:bmModCont}, we may bound $|f_t(x)- f_t(\theta_n)|$ by $\sigma |B(x)-B(\theta_n)|+ \max\{|a_{+}|,|a_{-}|\}$ for any $x\in [\theta_{n}, \theta_{n+1}]$ \red{(since $|x-\theta_n| \leq 1$ and $x \theta_n \geq 0$)}. As a consequence, for all $s>\max\{|a_{+}|,|a_{-}|\}$,
\begin{align*}
\mathbb{P}\big(\sup_{x\in [\theta_n,\theta_{n+1}]}|f_t(x)- f_t(\theta_n)|\geq s\big)\leq \mathbb{P}\big(\sigma \sup_{x\in [\theta_n,\theta_{n+1}]}|B(x)- B(\theta_n)|\geq s- \max\{|a_{+}|,|a_{-}|\}\big)\leq e^{-c(s-\max\{|a_{+}|,|a_{-}|\})^2},
\end{align*} 
where $c$ is a constant which does not depend on $n$ or $t$. The last inequality follows by applying reflection principle and tail decay of a Gaussian random variable. This shows \eqref{eq:ModCont}.
\bigskip
\\

\section{Auxiliary Results}\label{sec:Aux}
\subsection{Proof of Proposition~\ref{prop:intermittldp}}\label{sec:Aux1}
To prove \eqref{eq:ldp1}, it suffices to show that for $s > \zeta$,
\begin{align}
\underbrace{\limsup_{t\to \infty}\frac{1}{t}\log \mathbb{P}\big(X(t) \, \geq\, st\big) \leq -\max_{p\in \mathbb{R}_{>0}}\{ps - h(p)\}}_{\mathfrak{LimSup}} , \qquad \underbrace{\liminf_{t\to \infty}\frac{1}{t} \,\log\,\mathbb{P}\big(X(t) \, \geq\, st\big) \geq -\max_{p\in \mathbb{R}_{>0}}\{ps - h(p)\}}_{\mathfrak{LimInf}}.  
\end{align}
We first show $\mathfrak{LimSup}$. Recall the definition of $h$. Note that $h'$ is strictly increasing and has a continuous inverse. Let us define $\mathfrak{q}: (0, \infty) \to (\zeta, \infty)$ as $\mathfrak{q}(s) := (h')^{-1} (s)$. Note that the supremum of $ps - h(p)$ is attained when $p $ is equal to $\mathfrak{q}(s)$ and therefore, 
$
\sup_{p > 0} \{ps - h(p)\} = \mathfrak{q}(s) s - h(\mathfrak{q}(s))
$. By using the Markov's inequality, we get $
\mathbb{P}(X(t) \, \geq\, st) \leq e^{-\mathfrak{q}(s) s t }  \mathbb{E}[e^{\mathfrak{q}(s) X(t)}]$. 
We take the logarithm of both sides of this inequality, divide them by $t$ and let $t \to \infty$. Consequently,
\begin{equation*}
\limsup_{t \to \infty} \frac{1}{t} \log\mathbb{P}\Big(X(t) \geq st\Big) \leq \limsup_{t\to \infty} \frac{1}{t} \Big( -\mathfrak{q}(s) s t+\log \mathbb{E}[e^{\mathfrak{q}(s) X(t)}]\Big)= - (s\mathfrak{q}(s)-h(\mathfrak{q}(s))),
\end{equation*}
where the last equality follows from \eqref{eq:temp1}. This proves $\mathfrak{LimSup}$.

We turn to show $\mathfrak{LimInf}$. 
To this aim, we define $\q_{\epsilon}:(0,\infty)\to (\zeta,\infty)$ as $\q_\epsilon(s) = (h')^{-1} (s+\epsilon)$. For convenience of notation, we will use $\q_{\epsilon}$ to denote $\q_{\epsilon}(s)$. Fix any $s,t>0$. We define a exponentially tilted probability measure $\widetilde{\mathbb{P}}_{t, s}$ as
\[\widetilde{\mathbb{P}}_{t,s}\big(X(t) \in A\big) := \frac{1}{\mathbb{E}\big[e^{\q_\epsilon X(t)} \big]} \mathbb{E}\big[e^{\q_\epsilon X(t)} \mathbbm{1}_{\{X(t) \in A\}}\big],\]
where $A$ is a Borel set in $\mathbb{R}$. We denote the expectation with respect to $\widetilde{\mathbb{P}}_{t,s}$ by $\widetilde{\mathbb{E}}_{t,s}$. 
We claim that for showing $\mathfrak{LimInf}$, it suffices to verify for any fixed $s>0$,
\begin{equation}\label{eq:temporary}
\lim_{t \to \infty} \widetilde{\mathbb{P}}_{t,s}\Big(X(t) \in [ts, t(s+2\epsilon)]\Big) = 1.
\end{equation}
Let us first explain how $\mathfrak{LimInf}$ follows from \eqref{eq:temporary}. 
From the definition of $\widetilde{\mathbb{P}}_{t,s}$, we know the following change of measure formula-
\begin{align}\label{eq:temp4}
\mathbb{P}(X(t) \geq ts) 
= \mathbb{\widetilde{E}}_{t,s}[e^{-\q_\epsilon X(t)} \mathbbm{1}_{\{X(t) 
\geq ts\}}] \cdot \mathbb{E}\Big[e^{\q_\epsilon X(t)}\Big].
\end{align}
 Since $\{ts \leq X(t) \leq t(s+2\epsilon)\}$ is contained in $\{X(t)\geq ts\}$, we get the following inequality
\begin{equation}\label{eq:temp7}
\widetilde{\mathbb{E}}_{t,s}\Big[e^{-\mathfrak{q}_\epsilon X(t)} \mathbbm{1}_{\{X(t) \geq ts\}}\Big] \geq \widetilde{\mathbb{E}}_{t,s}\Big[e^{-\mathfrak{q}_\epsilon X(t)} \mathbbm{1}_{\{ts \leq  X(t) \leq t(s+2\epsilon)\}}\Big] \geq e^{-(s+2\epsilon) t\mathfrak{q}_\epsilon } \widetilde{\mathbb{P}}_{t,s}\Big(ts \leq X(t)\leq t(s+2\epsilon) \Big).
\end{equation}
Substituting this inequality into the right hand side of \eqref{eq:temp4} yields 
\begin{equation}\label{eq:temp8}
\mathbb{P}\Big(X(t) \geq ts\Big) \geq e^{-(s + 2\epsilon) t\mathfrak{q}_\epsilon } \mathbb{E}\Big[e^{\mathfrak{q}_\epsilon X(t)}\Big] \widetilde{\mathbb{P}}_{t,s}\Big(st \leq X(t) \leq (s+2\epsilon)t\Big).
\end{equation}


We take the logarithm of both sides of the above inequality and divide them by $t$ for both sides of \eqref{eq:temp8}. Letting $t \to \infty$, we conclude that  
\begin{equation*}
\liminf_{t \to \infty} \frac{1}{t}\log\mathbb{P}\Big(X(t) \geq ts\Big) \geq -(s+2\epsilon)\mathfrak{q}_\epsilon  + \liminf_{t \to \infty}\frac{1}{t} \mathbb{E}\Big[e^{\mathfrak{q}_\epsilon X(t)}\Big] = -(s+2\epsilon) \mathfrak{q}_\epsilon + h(\mathfrak{q_\epsilon}),
\end{equation*}
where the first inequality follows from \eqref{eq:temporary} and the second equality follows from \eqref{eq:temp1}. Recall that $\lim_{\epsilon \to 0} \mathfrak{q}_\epsilon = \mathfrak{q}(s)$. By the continuity of $h$, as $\epsilon \to 0$, the right hand side in the above display converges to $-s\mathfrak{q}(s)  + h(\mathfrak{q}(s))$. Recall that $-s\mathfrak{q}(s)  + h(\mathfrak{q}(s))$ is equal to $-\max_{p\in \mathbb{R}_{>0}}\{sp-h(p)\}$. This completes demonstrating how $\mathfrak{LimInf}$ follows from \eqref{eq:temporary}. Throughout the rest, we prove \eqref{eq:temporary}.
\smallskip
\\

In order to prove \eqref{eq:temporary}, it is enough to demonstrate $\lim_{t \to \infty} \widetilde{\mathbb{P}}_{t,s}\big(X(t) \notin [ts, t(s+2\epsilon)]\big) = 0$. This follows from the combination of  following results: 
\begin{align}\label{eq:duo}
\limsup_{t \to \infty} \frac{1}{t}\log \widetilde{\mathbb{P}}_{t,s}(X(t) < ts) < 0, \qquad \limsup_{t \to \infty} \frac{1}{t}\log \widetilde{\mathbb{P}}_{t,s}\big(X(t) > t(s + 2\epsilon)\big) < 0.
\end{align}
We proceed to prove these below. We first show $\lim_{t \to \infty} t^{-1}\log \widetilde{\mathbb{P}}_{t,s}(X(t) < ts) < 0$. By Markov's inequality, for $\lambda > 0$,
\begin{align*}
\widetilde{\mathbb{P}}_{t,s}(X(t) < ts) &\leq e^{\lambda s t} \widetilde{\mathbb{E}}_{t,s}[e^{-\lambda X(t)}] = e^{\lambda s t} \frac{\mathbb{E}[e^{(\q_\epsilon-\lambda) X(t)}]}{\mathbb{E}[e^{\q_\epsilon X(t)}]}.
\end{align*}
We take the logarithm of both sides and divide them by $t$. Letting $t \to \infty$ and utilizing \eqref{eq:temp1}, we get 
\begin{align*}
\limsup_{t \to \infty} \frac{1}{t} \log \widetilde{\mathbb{P}}_{t,s}\Big(X(t) < ts\Big) &\leq \lambda s + h(\q_\epsilon - \lambda) - h(\q_\epsilon).
\end{align*}
The desired result will follow from the above inequality if we can find a positive $\lambda$ such that the right hand side above is negative. To find such $\lambda$, we consider $H:(0,\infty)\to \mathbb{R}$ as $H(\lambda) := \lambda s + h(\mathfrak{q}_\epsilon - \lambda) -  h(\mathfrak{q}_\epsilon)$. It is straightforward that $H (0) = 0$ and $H'(0) = s - h'(\mathfrak{q}_\epsilon) = -\epsilon < 0$ since $\mathfrak{q}_\epsilon := (h')^{-1} (s + \epsilon)$. Since $H$ has continuous derivative, there exists $\lambda^{*} > 0$ such that $H(\lambda^{*}) < 0$. This implies $
\limsup_{t \to \infty} t^{-1} \log \widetilde{\mathbb{P}}(X(t) < ts) < 0
$ which concludes the desired result.
\smallskip
\\

Now we show $\limsup_{t \to \infty} t^{-1}\log \widetilde{\mathbb{P}}_{t,s}(X(t) > t(s + 2\epsilon)) < 0$. By Markov's inequality, for $\lambda > 0$,
\begin{align*}
\widetilde{\mathbb{P}}_{t,s}\big(X(t) > t(s + 2\epsilon)\big) \leq e^{-\lambda (s + 2\epsilon) t} \widetilde{\mathbb{E}}_{t,s}\big[e^{\lambda X(t)}\big] = e^{-\lambda t(s + 2\epsilon) } \frac{\mathbb{E}[e^{(\q_\epsilon+\lambda) X(t)} ]}{\mathbb{E}[e^{\q_\epsilon X(t)}]}.  
\end{align*}
In the same way as in the previous case, we get
\begin{align*}
\limsup_{t \to \infty} \frac{1}{t} \log \mathbb{P}\Big(X(t) > t(s + 2\epsilon)\Big) &\leq -\lambda (s + 2\epsilon) + h(\q_\epsilon + \lambda) - h(\q_\epsilon).
\end{align*}
To conclude the desired result, we will find $\lambda>0$ such that the right hand side is less than $0$. Like as in before, we consider $\widetilde{H}:(0,\infty)\to \mathbb{R}$ as $\widetilde{H} (\lambda) := -\lambda (s + 2\epsilon) + h(\q_\epsilon + \lambda) - h(\q_\epsilon)$ for which we know that $\widetilde{H} (0) = 0$, $\widetilde{H}'(0) = -s - 2\epsilon + h'(\mathfrak{q}_\epsilon) = -\epsilon$. Combination of these observations with the continuity of $\widetilde{H}^{\prime}$ yields the existence $\lambda>0$ such that the right hand side of the above inequality is less than $0$. This proves the second limiting result of \eqref{eq:duo} and hence, completes the proof.

\subsection{Proof of Lemma~\ref{lem:condition}}\label{sec:Aux2}

For proving (i), we first note that the logarithm of $ M^{f_t}_p(t,x)$ is a convex function of $p\in (0,\infty)$ which can be checked by verifying that the second derivative of $\log M^{f_t}_p(t,x)$ w.r.t. $p$ stays positive for all $p\in(0,\infty)$. Since the convexity is preserved under taking pointwise supremum and/or, limit of a sequence of convex functions, the convexity of $g(p)$ for $p\in (0,\infty)$ now follows from its definition and the fact that $\log M^{f_t}_p(t,x)$ is convex in $p$. To see the non-negativity of $g(p)$, for $t > T_0$, we write 
\begin{equation*}
\sup_{x \in \mathbb{R}} \Big(\frac{-p x^2}{2t} + \log M^{f_t}_p(t,x)\Big) \geq -\frac{pL^2}{2t} + \sup_{x \in [-L, L]}\log M^{f_t}_p(t,x) \geq -\frac{pL^2}{2t} - K,
\end{equation*}
where the first inequality follows noting that the function $\frac{-p x^2}{2t} $ takes its minimum value in the interval $[-L,L]$ at $\pm L$ and the second inequality is obtained by applying the lower bound condition \eqref{eq:lower bound condition} on $f_t$. By dividing both sides of the above inequality by $t$ and letting $t$ go to $\infty$, the limit of the left hand side yields $g(p)$ whereas the right hand side goes to $0$. This proves that $g(p)\geq 0$ for all $p > 0$.
\bigskip
\\
We turn to show (ii). We first prove that for every $p > 0$ and $\gamma > 0$, the set $\mathrm{MAX}^{f}_{p, \gamma}(t)$ is nonempty. By the definition of supremum, it suffices to prove $\sup_{y \in \mathbb{R}} \{-\frac{py^2}{2t} + \log M_{p}^{f_t} (t, y)\}$ is finite. By the growth condition \eqref{eq:growth condition}, we know that for all $t > 0$,
\begin{equation*}
-\frac{p x^2}{2t} + \log M^{f_t}_p(t,x) \leq - \frac{px^2}{2t} + C |x| + \frac{\alpha p x^2}{2t} + C = \frac{p(\alpha - 1) x^2}{2t} + C|x| + C.
\end{equation*}
Since $\alpha < 1$, the supremum of the right hand side over $x \in \mathbb{R}$ is finite, which implies $\sup_{y \in \mathbb{R}} \{-\frac{py^2}{2t} + \log M_p^{f_t} (t, y)\}$ is finite. This shows $\mathrm{MAX}^{f}_{p, \gamma}(t)$ is nonempty.
 
\smallskip
It remains to show that for fixed $p, \gamma > 0$, there exists $T_0$ and $C = C(p, \gamma)$, $|x_{q, \gamma}(t)|\leq Ct$ for all $t > T_0$ and $\frac{p}{2}< q< 2p$. We prove this by contradiction. Suppose that $|x_{q, \gamma}(t)|$ exceeds $\frac{4C t}{q(1-\alpha)}$ for some $q\in [p/2,2p]$ where $C$ is the constant in the growth condition \eqref{eq:growth condition} of $f_t$. In this occasion, we will show that the lower bound to $-qx^2_{q,\gamma}(t)/2t+ M^{f_t}_{q}(t,x_{q,\gamma}(t))$ has to exceeds its upper bound. Below, we separately compute an upper bound and a lower bound to $-qx^2_{q,\gamma}(t)/2t+ M^{f_t}_{q}(t,x_{q,\gamma}(t))$. Before proceeding to those computations, we note that for any $\frac{p}{2} \leq q \leq 2p$, $x\in \mathbb{R}$ and $t\in \mathbb{R}_{>0}$ (using H\"{o}lder's inequality when $f_t(x)$ is random),
\begin{align}
-\frac{q x^2}{2t} + \frac{2q}{p} \log M^{f_t}_{p/2}(t,x)\leq -\frac{q x^2}{2t} + \log M^{f_t}_{q}(t,x) &\leq -\frac{q x^2}{2t} + \frac{q}{2p} \log M^{f_t}_{2p}(t,x). \label{eq:temp3} 
\end{align}

\noindent \textsc{Upper bound to $-qx^2_{q,\gamma}(t)/2t+ M^{f_t}_{q}(t,x_{q,\gamma}(t))$:}
By the growth condition \eqref{eq:growth condition}, the right hand side of the second inequality in \eqref{eq:temp3} is bounded above by $-q (1-\alpha)x^2/2t + C|x|$. Plugging the bound on $|x_{q,\gamma}(t)|$ into this bound shows that $-q x^2_{q, \gamma}(t)/2t + \log M^{f_t}_{q}(t,x_{q,\gamma}(t))$ is bounded above by the maximum of $-q (1-\alpha)x^2/2t + C|x|$ over $x \in
 \mathbb{R}$, which equals $-\frac{2C^2 t}{q (1-\alpha)}$.
\smallskip

\noindent \textsc{Lower bound to $-qx^2_{q,\gamma}(t)/2t+ M^{f_t}_{q}(t,x_{q,\gamma}(t))$:} We claim and prove that $-qx^2_{q,\gamma}(t)/2t+ M^{f_t}_{q}(t,x_{q,\gamma}(t))$ is bounded below by $-\frac{q L^2}{2t} - 4K -\delta$ where $K$ is the same constant as in the lower bound condition \eqref{eq:lower bound condition} of $f_t$. 

Recall that $x_{q, \gamma} (t) \in \mathrm{MAX}^{f}_{q, \gamma}(t)$. Referring to \eqref{eq:MAXDef}, 
\begin{equation*}
-\frac{q x_{q, \gamma} (t)^2}{2t} + \log M^{f_t}_{q}(t,x_{q,\gamma}(t))  \geq \sup_{y \in \mathbb{R}} \Big\{-\frac{qy^2}{2t} + \log M^{f_t}_{q}(t, y)\Big\} - \gamma .
\end{equation*}
Due to the first inequality of \eqref{eq:temp3}, $-\frac{qy^2}{2t} + \log M^{f_t}_{q}(t, y)$ is bounded below by $-qy^2/2+ 2p^{-1}q \log M^{f_t}_{p/2}(t,y)$ for all $y \in \mathbb{R}$. Substituting this inequality into the right hand side of the above display and restricting the supremum over the interval $[-L,L]$, we get  
\begin{equation*}
-\frac{q x_{q, \gamma} (t)^2}{2t} + \log \mathbb{E}\Big[e^{q f_t (x_{q, \gamma} (t))}\Big] \geq \max_{y \in [-L, L]} \Big\{-\frac{q y^2}{2t} + \frac{2q}{p} \log M^{f_t}_{p/2}(t,y)\Big\} - \gamma,
\end{equation*}
where the constant $L$ is same as in the lower bound condition \eqref{eq:lower bound condition} for $\log M^{f_t}_{p/2}(t,y)$. One may bound the right hand side of the above display from below by $-qL^2/2t+ 2p^{-1}q\max_{y\in [-L,L]} \{\log M_{p/2}(t,y)\} - \gamma$. Since $2p^{-1}q\max_{y\in [-L,L]} \{\log M_{p/2}(t,y)\}$ bounded below by $-4K$ due to \eqref{eq:lower bound condition} and $q < 2p$, we find the right hand side in the above display is lower bounded by $-\frac{q L^2}{2t} - 4K -\gamma$.

As we have shown above, if $|x_{q, \gamma} (t)| >  \frac{4C t}{q (1-\alpha) }$, our lower bound to $-\frac{q x_{q,\gamma} (t)^2}{2t} + \log M^{f_t}_q(t,z)$ (which is $-\frac{q L^2}{2t} - 4K -\delta$) exceeds the upper bound (which is $-\frac{2C^2 t}{q (1-\alpha)}$) for all large $t$. This is a contradiction. Hence, the result follows.

\subsection{Proof of Proposition~\ref{prop:convolution}}
\label{sec:Aux3}
The proof is essentially the same as that in Lemma 1.18 of \cite{CH16}. We add more detail. First, we note that 
\begin{equation*}
\mathcal{Z}^{\textbf{nw}, y} (t, x) = p(t, x-y) + \int_0^t \int_{\mathbb{R}} p(t-s, x-z) \mathcal{Z}^{\textbf{nw}, y}(s, z) \xi(s, z) dsdz.
\end{equation*}
where $\mathcal{Z}^{\textbf{nw}, y}$ denotes the solution of the SHE started from the delta initial measure at $y$. Multiplying both sides of the above display by $e^{f(y)}$ and integrating with respect to $y$ over $\mathbb{R}$ shows that $\int_\mathbb{R} e^{f(y)} \mathcal{Z}^{\textbf{nw}, y} (t, x) dy$ satisfies the mild equation of the SHE starting from $e^f$. For any fixed $x \in \mathbb{R}$, the spatial process of $\mathcal{Z}^{\textbf{nw}, y} (t, x)$ in $y$ has the same distribution as $\mathcal{Z}^{\textbf{nw}, x} (t, y)$ by the ``time reversal property" of the SHE. Finally, $\mathcal{Z}^{\textbf{nw}, x} (t, y)$ has the same distribution as  $\mathcal{Z}^{\textbf{nw}} (t, y-x)$ due to the translation invariance of the space-time white noise and the latter has the same distribution as $\mathcal{Z}^{\textbf{nw}} (t, x-y)$ as a process in $y$ due to symmetry. Since $f$ is independent of $\xi$, the law of $\int_\mathbb{R} e^{f(y)} \mathcal{Z}^{\textbf{nw}, y} (t, x) dy$ (as a process in $x$) remains same after replacing $\mathcal{Z}^{\textbf{nw}, y} (t, x)$ with $\mathcal{Z}^{\textbf{nw}} (t, x-y)$ for any fixed $x$. This shows $\int_\mathbb{R} e^{f(y)} \mathcal{Z}^{\textbf{nw}} (t, x-y) dy$ is the mild solution of the SHE started from from $f$ for any fixed $x\in \mathbb{R}$. This completes the proof of Proposition \ref{prop:convolution}.
\bibliographystyle{alpha}
\bibliography{ref}

\end{document}